\documentclass[10pt]{article}

\usepackage[top = 1.5 in, bottom = 1.5in, left = 1.5 in, right = 1.5 in]{geometry}

\input{mymacros.sty}
	
\begin{document}

\title{Local models for Weil-restricted groups}
    \author{Brandon Levin}

\maketitle

\begin{abstract}
We extend the group theoretic construction of local models of Pappas and Zhu \cite{PZ} to the case of groups obtained by Weil restriction along a possibly wildly ramified extension.  This completes the construction of local models for all reductive groups when $p \geq 5$.  We show that the local models are normal with special fiber reduced and study the monodromy action on the sheaves of nearby cycles. As a consequence, we prove a conjecture of Kottwitz that the semi-simple trace of Frobenius gives a central function in the parahoric Hecke algebra.  We also introduce a notion of splitting model and use this to study the inertial action in the case of an unramified group.         
\end{abstract}

\tableofcontents

\section{Introduction}

A \emph{local model} is a projective scheme over the ring of integers of a $p$-adic field $E$ which is supposed to \'etale locally model the integral structure of a Shimura variety with parahoric level structure. The theory of local models for Shimura varieties of PEL-type was developed in \cite{RZ}. There were subsequent refinements studied mostly on a case by case basis by Faltings, G\"ortz, Haines, Pappas, and Rapoport, among others. Pappas and Zhu \cite{PZ} give a new construction of local models which unlike previous constructions, is purely group-theoretic, i.e., it does not rely on any particular representation of $G$. They build their local models inside degenerations of affine Grassmanians, extending constructions of Beilinson, Drinfeld, Gaitsgory, and Zhu to mixed characteristic. This allows them to employ powerful techniques from geometric representation theory. In particular, the coherence conjecture of Pappas and Rapoport, proved in \cite{ZhuCoh}, plays a crucial role in understanding the geometry of their local models.  

Specifically, \cite{PZ} associates a local model to triples $(G, \cP, {\mu})$ where $G$ is a connected reductive group over a $p$-adic field $F$, $\cP$ is a parahoric (level) subgroup of $G(F)$, and $\{ \mu \}$ is a geometric conjugacy class of cocharacters of $G$.  Their approach requires an assumption on the group $G$, that it splits over a tamely ramified extension of $F$. If $G$ is absolutely simple, this is mild assumption which is always satisfied if $G$ is simply-connected or adjoint and $p \geq 5$. However, it is easy to construct non-absolutely simple groups which do not satisfy this tameness hypothesis by taking a wildly ramified extension $K/F$ and considering the Weil-restriction $\Res_{K/F} G$ of any reductive group $G$ over $K$. Furthermore, these sorts of groups arise naturally in the theory of Shimura varieties whenever one works over a field other than $\Q$. For example, Hilbert modular varieties are associated to the group $\Res_{L/\Q} \GL_2$ for a totally real field $L/\Q$. 

In this paper, we extend Pappas and Zhu's construction to groups of the form $\Res_{K/F} G$ where $G$ is any connected reductive group over $K$ which splits over a tame extension of $K$. If $p \geq 5$, then any group $G'/F$ is isogenous to a product of groups of this form and so this completes the construction of local models except when $p = 2$ or $3$ (see Remark \ref{isogeny}). In the case where $G = \GL_n$ or $\GSp_{2g}$, these local models were constructed and studied in \cite{PR1, PR2}.  If $K/F$ is a tame extension, our local models agree with those constructed by Pappas and Zhu (see Proposition \ref{PZcomp}). Our construction is partially inspired by the splitting models introduced in \cite{PR2}.

Let $\cP$ be a parahoric subgroup of $G(K) = (\Res_{K/F} G)(F)$, and let $\{ \mu \}$ be a geometric conjugacy class of cocharacters of $\Res_{K/F} G$. Denote by $E$ the minimal field of definition (\emph{reflex field}) of $\{\mu \}$.  The (generalized)\footnote{This refers to the fact that we do not assume here that the cocharacter $\mu$ is minuscule. We will usually omit the adjective ``generalized'' later on.} local model $M(\mu)$ is a projective scheme over the ring of integers $\cO_{E}$ of $E$. Our first main result is the analogue of Theorem 8.1 in \cite{PZ}:
\begin{thm} \label{Intlocmodel} Suppose that $p \not{\mid} \: |\pi_1(G_{\text{der}})|$. Let $k_E$ be the residue field of $\cO_E$. The scheme $M(\mu)$ is normal.  The special fiber $M(\mu) \otimes k_E$ is reduced, normal, and Frobenius split. Each irreducible component of $M(\mu) \otimes k_E$ is Cohen-Macaulay.   
\end{thm} 

As in \cite{PZ}, one can also identify the special fiber as a union of affine Schubert varieties in an affine flag variety (see Theorem \ref{thmadm} for a precise statement).  The main ingredients in the proof of Theorem \ref{Intlocmodel} are the coherence conjecture of Pappas and Rapoport proved in \cite{ZhuCoh} and a product formula which is also found in \cite{ZhuCoh}.   The tameness hypothesis in \cite{PZ} appears in their construction of a certain global affine Grassmanian inside which the local models are built.  The affine Grassmanians and affine flag varieties 
which appear in \cite{PZ} are mixed characteristic versions of now familiar constructions in the theory of the geometric Langlands correspondence.  In this same spirit, we adapt their construction to take into account the Weil restriction from $K/F$ building on ideas of Beilinson and Drinfeld in equal characterstic \cite{BD}. We now discuss this construction in some detail since the other arguments of \cite{PZ} carry over mostly formally to our situation. 

Assume for simplicity that $K/F$ is totally ramified.  In \cite[\S 4]{PZ}, they associate to the pair $(G, \cP)$ over $K$ a parahoric or Bruhat-Tits group scheme $\cG$ over $\bA^1_{\cO_K} = \Spec \cO_K[u]$. The group $\cG$ is a smooth affine group scheme with connected fibers which is reductive over $(\Gm)_{\cO_K}$ and ``parahoric'' at $u = 0$. This allows them to define a degeneration $\Gr_{\cG, \cO_K}$ of the affine Grassmanian $\Gr_G$ over $K$  whose special fiber is an affine flag variety over the residue field $k_0$ of $\cO_K$ which is determined by $\cP$ (up to some choices).  This is a mixed characteristic version of constructions of Gaitsgory and Beilinson-Drinfeld.  Note that we have specialized Pappas and Zhu's family $\Gr_{\cG, X}$ to $\Spec \cO_K$ via $u \mapsto \varpi$, where $\varpi$ is a uniformizer of $K$.  In our situation, we desire a degeneration of $\Gr_{\Res_{K/F} G}$ over $F$ to $\cO_F$ whose special fiber is again an affine flag variety determined by $\cP$.  Geometrically $\Gr_{\Res_{K/F} G}$ is $[K:F]$-copies of the affine Grassmanian for $G$ and so we are reminded of the BD-Grassmanian (\cite{BD, RicharzAG}) in equal characteristic which degenerates $d$ copies of the affine Grassmanian to a single affine Grassmanian.  Our construction combines a mixed characteristic version of the BD-Grassmanian with the construction from \cite{PZ}.  

Choose a uniformizer $\varpi$ of $K$.  Let $Q(u)$ be the minimal polynomial for $\varpi$ over $F$.   Our version of $\Gr_{\cG, \cO_K}$ from \cite{PZ} will be a moduli space $\Gr^{Q(u)}_{\cG}$ over $\cO_F$ of $\cG$-bundles over $\bA^1_{\cO_F}$ with a trivialization away from $Q(u) = 0$.  When $p$ is invertible, $Q(u)$ decomposes geometrically into a disjoint union of $[K:F]$-divisors and this gives rise to $\Gr_{\Res_{K/F} G}$.  Over the residue field $k$ of $\cO_F$, however, the roots of $Q(u)$ ``collide'' and we see an affine flag variety for the completion of $\cG$ at $k[\![u]\!]$ which by construction will be a parahoric group scheme $\cP^{\flat}$ over $k[\![u]\!]$. There is a technical subtlety related to the construction of the Bruhat-Tits group scheme $\cG$.  The parahoric group scheme in \cite{PZ} is defined over $\bA^1_{\cO_K}$ whereas it is important that we have a moduli problem over $\cO_F$.  Using that the group $G/K$ is tame, we are able to arrange so that $\cG$ is in fact defined over $\bA^1_{\cO_F}$.  This is done in \S 3.  

If $E$ is the reflex field of $\{\mu\}$, then the local model $M_{\cG}(\mu)$ is defined to be the closure of the affine Schubert variety $S(\mu) \subset (\Gr_{\Res_{K/F} G})_E$ in $\left(\Gr^{Q(u)}_{\cG} \right)_{\cO_{E}}$.  Since $M_{\cG}(\mu)$ is defined by Zariski closure, it does not have a simple moduli space description.  However, the interpretation as a subvariety of $\Gr^{Q(u)}_{\cG}$ allows us to prove many nice properties of $M_{\cG}(\mu)$ as in Theorem \ref{Intlocmodel}.   This connection with affine Grassmanian also allows us to study the sheaf of nearby cycles  $\RPsi (\IC_{\mu})$\footnote{When $\mu$ is minuscule, $S(\mu)$ is smooth and so $\IC_{\mu}$ is the constant sheaf $\overline{\Q}_{\ell}$ up to a shift.} in \S 5.   Whenever $M_{\cG}(\mu)$ \'etale locally describes the structure of an integral model of a Shimura variety, one then obtains results for the nearby cycles sheaf for the Shimura variety with parahoric level structure. 

Assume for simplicity of statement that $\mu$ is a minuscule cocharacter.  Our group scheme $\cG$ gives rise to a group $G^{\flat} := \cG_{k(\!(u)\!)}$ where $k$ is residue field of $F$ and a parahoric group scheme $\cP^{\flat} := \cG_{k[\![u]\!]}$. For any $\F_q \supset k_E$, we can view the semisimple trace of Frobenius $\tau^{\mathrm{ss}}_{\RPsi(\overline{\Q}_{\ell})}$ as an element of the parahoric Hecke algebra $\cH_q(G^{\flat}, \cP^{\flat})$ of bi-$\cP^{\flat}(\F_q(\!(u)\!))$-equivaraint, compactly supported locally constant $\overline{\Q}_{\ell}$-valued functions on $G^{\flat}(\F_q(\!(u)\!))$.  
\begin{thm} $($Kottwitz's conjecture$)$ The semi-simple trace of Frobenius on the sheaf of nearby cycles $\RPsi (\overline{\Q}_{\ell})$ of $M_{\cG}(\mu) \ra \Spec \cO_E$ gives a central function in the parahoric Hecke algebra $\cH_q(G^{\flat}, \cP^{\flat})$. 
\end{thm}
See \S 5.3 and Theorem \ref{Kottwitz} for more details and a more general statement.  The proof of this theorem follows a strategy of \cite{HNnearby} and \cite{Gaitsgory} to prove a commutativity constraint on the nearby cycles (see Theorem \ref{commconst}).

If $\widetilde{F}$ is an extension of $F$ over which $\Res_{K/F} G$ becomes split, then we show that the inertial action of $I_{\widetilde{F}}$ on the the nearby cycles is unipotent generalizing results of \cite{HNnearby} (see Theorem \ref{monodromy}).  Moreover, when the parahoric is \emph{very special} in the sense \cite[Definition 6.1]{ZhuGS} and $G$ is an unramified group over $K$, we describe the sheaf of nearby cycles as a convolution product (see Theorem \ref{convid} for a precise statement). We also give a conjectural description of the action of $I_E$.  This generalizes a conjecture of \cite{PR1} in the case of $G = \GL_n$ or $G = \GSp_{2g}$.  The key tool is an affine Grassmanian version of the splitting models introduced in \cite{PR2}.  These are closely related to the convolution Grassmanians from geometric Langlands.  

We now give a brief overview of the sections. In \S 2, we introduce local models in the special case where $K/F$ is totally ramified and $G$ is an unramified group over $F$.  In this case, the constructions more closely resemble those of \cite{Gaitsgory}. In \S 3, we do the technical work of constructing the parahoric (or Bruhat-Tits) group scheme $\cG$ over $\cO_F[u]$ attached to the group $G/K$ and a parahoric subgroup of $G(K)$. 

In \S 4, we construct the local models associated to the conjugacy class $\{\mu\}$ and prove the main theorem (Theorem \ref{Intlocmodel}) on their geometry.  We also prove the ind-representability of the affine Grassmanian $\Gr_{\cG}^{Q(u)}$  and show that it is ind-proper.  Our proof is different from \cite{PZ} and follows a strategy introduced in \cite{RicharzAG}.   In the last section (\S5), we study the nearby cycles on local models.  To do this, we introduce a number of auxiliary affine Grassmanians including versions of convolution and BD-Grassmanians, and what we call a \emph{splitting Grassmanians} which generalize the splitting models of \cite{PR2}.   

\subsection{Acknowledgments}  
This paper is partially based on the author's Ph.D. thesis advised by Brian Conrad to whom the author owes a debt of gratitude for his generous guidance.  It is a pleasure to thank Xinwen Zhu and George Pappas for many helpful exchanges and for comments on an earlier draft. I would also like to thank Tasho Kaletha and Bhargav Bhatt for useful discussions related to the project.  Part of this work was completed while the author was a visitor at the Institute for Advanced Study supported by the National Science Foundation grant DMS-1128155.  The author is grateful to the IAS for its support and hospitality.  The debt this paper owes to the work of Michael Rapoport, George Pappas and Xinwen Zhu will be obvious to the reader.   
  
\subsection{Notation}
Let $F$ denote a finite extension of $\Qp$ with ring of integers $\cO$.   We will always take $K$ to be a finite extension of $F$ and $K_0$ to be the maximal subfield of $K$ unramified over $F$.  We will use $\varpi$ to denote a uniformizer of $K$.  Set $d = [K:F]$.  Let $k$ denote the residue field of $\cO$ and $k_0/k$ the residue field of $K_0$. We will denote the minimal polynomial of $\varpi$ over $K_0$ by $Q(u)$.  For any $\cO_{K_0}$-algebra $R$, we will take $\widehat{R}_{Q(u)}$ to be the $Q(u)$-adic completion of $R[u]$.

Unless stated otherwise, $G$ will be a connected reductive group over $K$ which splits over a tame extension of $K$. We will fix such a tame extension $\widetilde{K}$ in \S 3. We use $\mu$ to denote a fixed cocharacter of $\Res_{K/F} G$ and denote its conjugacy class by $\{ \mu \}$. We will use $E$ for the reflex field of $\{ \mu \}$ inside a fixed algebraic closure $\overline{F}$ of $F$.  Parahoric subgroups will be denoted $\cP_*$, and abusing notation, we will be use the same notation for the corresponding parahoric group scheme. On the other hand parabolic subgroup (schemes) will be denoted by $P_*$.   

For any connected reductive group $H$ over a field $\kappa$, $H_{\der}$ will denote the derived subgroup and $H_{\ad}$ the adjoint quotient.   

\section{Unramified groups}  

In this section, we study local models for groups of the form $\Res_{K/F} G$ where $G$ is an unramified group defined over $F$. For simplicity, we will also assume that $K/F$ is totally ramified. This is special case of the what we consider in the rest of the paper.  The unramified case illustrates the main conceptual points without as many of the technical difficulties.  The construction in this section closely resembles the equal characteristic construction in \cite{Gaitsgory}.  In Proposition \ref{comp2}, we compare with the more general construction given in \S 4.    

\subsection{Preliminaries}

Let $H$ be any connected reductive group over a field $\kappa$.  The affine Grassmanian $\Gr_H$ is the moduli space of isomorphism classes of $H$-bundles on the formal unit disc $\Spec R[\![u]\!]$ together with a trivialization over the punctured unit disc $\Spec R(\!(u)\!)$ where $R$ is any $\kappa$-algebra.  As a functor on $\kappa$-algebras, $\Gr_H$ is represented by an ind-projective scheme over $\Spec \kappa$.

Let $\{\mu\}$ denote a geometric conjugacy class of cocharacters of $H$.  The smallest field of definition of $\{\mu\}$ inside a fixed algebraic closure $\overline{\kappa}$ of $\kappa$ is called the \emph{reflex field} of ${\mu}$. Here we denote it by $\kappa_{\{\mu\}}$.  The affine Schubert variety $S_H(\mu)$ associated to $\{\mu\}$ is a closed finite type subscheme of $(\Gr_H)_{\kappa_{\{\mu\}}}$.  It is the closure of an orbit for the positive loop group $L^+ H := H(\kappa[\![u]\!])$ acting on $\Gr_H$.

More generally, one can associate to any facet $\fa$ in the building $\cB(H_{\kappa(\!(u)\!)}, \kappa(\!(u)\!))$  an affine flag variety $\Fl_{\fa}$, an ind-projective scheme over $\kappa$ (cf. \cite[\S 1]{PRTwisted} or \cite[\S 1.1]{RicharzAG}).   Assume for the moment that $\kappa$ is algebraically closed. If $T_H$ is a maximal split torus of $H$, then the Iwahori-Weyl group is given by
$$
\widetilde{W} := \widetilde{W}(H, T_H, \kappa(\!(u)\!)) = N_H(T_H)(\kappa(\!(u)\!)) /T_H(\kappa[\![u]\!]).
$$
It fits into an exact sequence 
$$
0 \ra X_*(T_H) \ra \widetilde{W} \ra W(H, T_H) \ra 1
$$
where $W(H, T_H)$ is the ordinary Weyl group. The Iwahori-Weyl group is discussed in a more general setting in \S 3.2.   

If $\cH_{\fa}$ is the parahoric group scheme over $\kappa[\![u]\!]$ associated to $\fa$, then the \emph{positive loop group} $L^+ \cH_{\fa}$ is the functor on $\kappa$-algebras defined by 
$$
L^+ \cH_{\fa}: R \mapsto \cH_{\fa}(R[\![u]\!]).
$$
The loop group $L^+ \cH_{\fa}$ acts on $\Fl_{\fa}$ and its orbits are indexed by elements of $\widetilde{W}$ (cf. \cite[Proposition 1.2]{RicharzAG}). 

\begin{defn}
 For any $w \in \widetilde{W}$, let $S_{w}^{\fa}$ be the associated closed affine Schubert variety in $\Fl_{\fa}$.  
\end{defn}

Recall that
$$
(\Fl_{\fa})_{\red} = \bigcup_{w \in \widetilde{W}} S_w^{\fa}.
$$
At the level of $\kappa$-points, this is related to the Cartan decomposition for $H(\kappa(\!(u)\!))$.
  

\subsection{Mixed characteristic affine Grassmanians}

Let $G$ be an unramified reductive group (quasi-split and split over an unramified extension) over $F$, and let $\fa$ be a facet of the building $\cB(G, F)$.  Assume furthermore that the closure of $\fa$ contains a hyperspecial vertex.  Hyperspecial vertices exist since $G$ is unramified.   Note that when $\fa$ is a chamber (i.e., when $\fa$ corresponds to an Iwahori subgroup) the closure always contains a hyperspecial vertex .

Pick a hyperspecial vertex $x_0$ in the closure of $\fa$.  The point $x_0$ corresponds to a connected reductive group scheme $\cG_0$ over $\Spec \cO$ whose generic fiber is $G$.  Furthermore, there is a unique parabolic $P_{k} \subset \cG_{k}$ such that
$$
\cP_{\fa} := \{ g \in \cG_0(\cO) \mid \overline{g} \in P_{k}(k) \}
$$
is the parahoric subgroup associated to $\fa$. Define $\cG := \cG_0 \otimes_{\cO} \cO[u]$, the constant extension. Observe that if we take $G^{\flat} := \cG_{k(\!(u)\!)}$ then $\cG_{k[\![u]\!]}$ is a reductive model of $G^{\flat}$ and the parabolic $P_{k}$ defines a parahoric subgroup
$$
\cP_{\fa^{\flat}} := \{ g \in \cG(k[\![u]\!]) \mid g \mod u \in P_{k}(k) \}
$$
of $G^{\flat}$ corresponding to a facet $\fa^{\flat}$ of $\cB(G^{\flat}, k(\!(u)\!)\!)$.  In particular, there is an affine flag variety $\Fl_{\fa^{\flat}}$ over $k$ associated to $\fa^{\flat}$ as in \S 2.1.  In section \S 3.3, we recall the general procedure from \cite{PZ} for producing $\fa^{\flat}$ from $\fa$. 

In the function field setting, \cite[\S 2.2.1]{Gaitsgory} constructs a deformation of the affine flag variety to the affine Grassmanian generalizing constructions of Beilinson and Drinfeld.  Specifically, there is an ind-projective scheme $\Fl_X$ over $\bA^1_k$ whose fiber at 0 is an affine flag variety (for Iwahori level) for a group $G^{\flat}$ and whose generic fiber is the product $G^{\flat}/B^{\flat} \times \Gr_{G^{\flat}}$.  Our goal is to construct a deformation of $\Fl_{\fa^{\flat}}$ to $\Spec \cO$ whose generic fiber is a product $G/P_F \times \Gr_{\Res_{K/F} G}$ for some parabolic $P_F$ of $G$. 

\begin{defn} Let $P$ be a parabolic subgroup scheme of $\cG$. If $\cE$ is a $\cG$-bundle on an $\cO$-scheme $X$, then a \emph{reduction} of $\cE$ to $P$ is a pair $(Q, i)$ where $Q$ is $P$-bundle on $X$ and $i:\cG \times^{P} Q \cong \cE$ is an isomorphism.  
\end{defn} 

\begin{exam} If $\cG = \GL_n$ and $B$ is the Borel, then a reduction of a vector bundle $V$ over $X$ to $B$ is same as a filtration of $V$ whose graded pieces are line bundles on $X$.  If $V$ is trivial, then a reduction to $B$ the same as a full flag.
\end{exam} 

Choose a parabolic subgroup $P \subset \cG$ which lifts $P_k \subset \cG_k$.  

\begin{defn} \label{bigflag} Let $Q(u) \in \cO[u]$ be the minimal polynomial of $\varpi$ over $F$. For any $\cO$-algebra $R$, define
$$
\Fl_{\cG}^{Q(u)}(R) := \{\text{isomorphism classes of triples } (\cE, \beta, \eps) \},
$$ 
where $\cE$ is a $\cG$-bundle over $R[u]$, $\beta$ is a trivialization of $\cE|_{\Spec (R[u])[1/Q(u)]}$, and $\eps$ is a reduction of $\cE \mod u$ to $P$.  If $P_{k} = \cG_{k}$, then $\Fl_{\cG}^{Q(u)}$ is the moduli of pairs $(\cE, \beta)$ which we denote by $\Gr_{\cG}^{Q(u)}$.
\end{defn}

\begin{rmk} The analogy with the Beilinson-Drinfeld Grassmanian \cite{BD} arises from the fact that over $\overline{F}$ the polynomial $Q(u)$ corresponds to $[K:F]$ distinct points whereas on the special fiber $Q(u) \equiv u^{[K:F]}$.  The bundle together with a reduction to $P$ is similar to the definition $\Fl_X$ from \cite[\S 2.2.1]{Gaitsgory}.
\end{rmk}

\begin{rmk} We have assumed that $K/F$ is totally ramified.  In the general case, if $K_0/F$ is the maximal unramified subextension, then one does the construction over $K_0$ and then takes the \'etale Weil restriction from $\cO_{K_0}$ to $\cO$. All of our main results easily reduce to the totally ramified case.
\end{rmk}

There is also a ``local" version of $\Gr_{\cG}^{Q(u)}$:
\begin{defn} \label{bigflagloc0} For any $\cO$-algebra $R$, let $\widehat{R}_{Q(u)}$ denote the $Q(u)$-adic completion of $R[u]$. Define
$$
\Gr_{\cG}^{Q(u), \mathrm{loc}}(R) := \{\text{iso-classes of pairs } (\cE, \beta) \},
$$ 
where $\cE$ is a $\cG$-bundle over $\widehat{R}_{Q(u)}$, $\beta$ is a trivialization of $\cE|_{\Spec \widehat{R}_{Q(u)}[1/Q(u)]}$.
\end{defn}

The natural map of functors $\Gr_{\cG}^{Q(u)} \ra \Gr_{\cG}^{Q(u), \mathrm{loc}}$ given by completion at $Q(u)$ is an equivalence.  This follows from the Beauville-Laszlo descent lemma for $\cG$-bundles (see \cite[Theorem 3.1.8]{LevinThesis} or \cite[Lemma 6.1]{PZ}).  We will use the two moduli $\Gr_{\cG}^{Q(u), \mathrm{loc}}$ and $\Gr_{\cG}^{Q(u)}$ interchangeably. 

\begin{exam} \label{GLN} Let $M_0 := \cO[u]^n$ be the trivial rank $n$ vector bundle over $\bA^1_{\cO}$.  A \emph{$Q(u)$-lattice} in $M_0[1/Q(u)]$ with coefficients in an $\cO$-algebra $R$ is a projective $R[u]$-submodule $M \subset (M_0 \otimes R)[1/Q(u)]$ such that $M[1/Q(u)] = (M_0 \otimes R)[1/Q(u)]$.  When $\cG = \GL_n$, $\Gr_{\GL_n}^{Q(u)}$ is the moduli space of $Q(u)$-lattices in $M_0[1/Q(u)]$.  The ind-scheme structure on $\Gr_{\GL_n}^{Q(u)}$ is given by the subfunctors of lattices lying between $Q(u)^{-N} M_0$ and $Q(u)^N M_0$ for $N \geq 0$.   See \cite[\S 10.1]{LevinThesis} for more details. 
\end{exam}

\begin{prop}  \label{fibers1} The functor $\Fl_{\cG}^{Q(u)}$ is represented by an ind-scheme which is ind-projective over $\Spec \cO$.  Furthermore,
\begin{enumerate}
\item The generic fiber $\Fl_{\cG}^{Q(u)}[1/p]$ is isomorphic to $(\cG_F/P_{F}) \times \Gr_{\Res_{K/F} G_F}$ over $\Spec F$;
\item The special fiber $\Fl_{\cG}^{Q(u)} \otimes_{\cO} k$ is isomorphic to the affine flag variety $\Fl_{\fa^{\flat}}$ corresponding to the group $G^{\flat}$ over $k(\!(u)\!)$ and the facet $\fa^{\flat}$ determined by $P_k$.
\end{enumerate}
\end{prop}
\begin{proof} The forgetful map $\Fl_{\cG}^{Q(u)} \ra \Gr^{Q(u)}_{\cG}$ is relatively representable and projective. The ind-scheme $\Gr^{Q(u)}_{\cG}$ can be shown to be ind-projective by choosing a faithful representation $\cG \ra \GL_n$ which induces a closed immersion $\Gr^{Q(u)}_{\cG} \iarrow \Gr^{Q(u)}_{\GL_n}$ (see \cite[Proposition 10.1.13]{LevinThesis}).  

By functorialities on $G$-bundles with respect to Weil restriction (\cite[Theorem 2.6.1]{LevinThesis}), the affine Grassmanian $\Gr_{\Res_{K/F} G_K}$ is naturally isomorphic to $\Res_{K/F} \Gr_{G_K}$.  Since $Q(u) \equiv u^{[K:F]} \mod m_{\cO}$, $\Fl_{\cG}^{Q(u)} \otimes_{\cO} k$ is exactly the affine flag variety $\Fl_{\fa^{\flat}}$ associated to $P_{k}$ as in \cite[\S 1.1.3]{Gaitsgory}.

For (1), observe that when $p$ is invertible, $Q(u)$ and $u$ define disjoint divisors on $\bA^1_F$.  Thus, for any $F$-algebra $R$ and any $(\cE_R, \beta_R, \eps_R) \in \Fl_{\cG}^{Q(u)}(R)$ the trivialization $\beta_R$ induces a trivialization of $\cE_R \mod u$.  Thus, $\eps_R$ is a reduction to $P_F$ on a trivial bundle, i.e., an $R$-point of $\cG_F/P_F$. 

Over a splitting field $L$ for $Q(u)$ over $F$, the pair $(\cE, \beta)$ is a $G$-bundle on $\bA^1_L$ trivialized away from the $[K:F] = d$ roots of $Q(u)$.  It is a standard fact that if $D = \coprod_{1 \leq i \leq d} D_i$ is a union of pairwise disjoint divisors then there is an equivalence between bundles with trivialization away from $D$ and $d$-tuples of bundles with trivializations away from $D_i$ respectively (see \cite[\S 6.2]{ZhuCoh} or the proof of Proposition \ref{genfiber}). Thus, over $L$, we get $[K:F]$-copies of the affine Grassmanian $\Gr_G$ centered at the roots of $Q(u)$. One concludes by Galois descent.        
\end{proof} 

\begin{rmk} The idea that local models should live in spaces like $ \Fl_{\cG}^{Q(u)}$ goes back to work of \cite{PR1, HNnearby}. Without the Weil restriction (so $K = F$), these sorts of deformations were constructed in \cite{PZ} and our presentation is a reformulation of theirs in this special case.  In \cite[\S 10]{LevinThesis}, we constructed a deformation of $\Gr_{\cG_k}$ to $\Gr_{\Res_{K/F} G}$ which is the special case where $P_{k} = \cG_{k}$, i.e., when $\fa$ is a hyperspecial vertex.  
\end{rmk}

Next we discuss the ``loop group" which acts on $\Fl_{\cG}^{Q(u)}$. Define a pro-algebraic group over $\Spec \cO$ by 
$$
L^{+, Q(u)} \cG (R) := \varprojlim_N \cG \left(R[u]/Q(u)^N \right) = \cG \left( \widehat{R}_{Q(u)} \right).
$$
For any $N \geq 1$, the functor $R \mapsto \cG \left(R[u]/Q(u)^N \right)$ is represented by the smooth affine group scheme $\Res_{(\cO[u]/Q(u)^N)/\cO} \cG$. 

Using the local description (Definition \ref{bigflagloc0}) of $\Gr^{Q(u)}_{\cG}$ , we see that $L^{+, Q(u)} \cG$ acts on pairs $(\cE, \beta)$ by changing the trivialization.  In the same way, $L^{+, Q(u)} \cG$ acts on $\Fl^{Q(u)}_{\cG}$ in such away that the forgetful map $\Fl^{Q(u)}_{\cG} \ra \Gr^{Q(u)}_{\cG}$ is equivariant. 

\begin{prop} \label{nice} The action of $L^{+, Q(u)} \cG$ is \emph{nice} in the sense of group actions on ind-schemes $($\cite[Appendix A]{Gaitsgory}$)$.
\end{prop}
\begin{proof}
The ind-scheme structure on $\Fl^{Q(u)}_{\cG}$ is compatible with the ind-scheme structure on $\Gr^{Q(u)}_{\cG}$ so we can reduce to the case of  $\Gr^{Q(u)}_{\cG}$.  One then reduces to $\GL_n$ using a faithful representation of $\cG$. The description of $\Gr^{Q(u)}_{\GL_n}$ in terms of $Q(u)$-lattices from Example \ref{GLN} makes it clear that the action is nice.  See \cite[Proposition 10.2.9]{LevinThesis} for details.
\end{proof} 

The ind-scheme $\Fl_{\cG}^{Q(u)}$ has a canonical point over $\Spec \cO$ given by the trivial bundle with its trivialization and reduction to $P$ which we denote by $1_{\Fl^{Q(u)}_{\cG}}$.  

\begin{defn} \label{parahoric10}  Let $L^{+, Q(u)} \cP$ be the stabilizer of  $1_{\Fl^{Q(u)}_{\cG}}$ in $L^{+, Q(u)} \cG$.  By Proposition \ref{nice}, $L^{+, Q(u)} \cP$ is a closed subgroup of $L^{+, Q(u)} \cG$.  
\end{defn}

\begin{prop} \label{fibers2} Let $L^{+, Q(u)} \cP$ be as above.  Then we have isomorphisms
$$
 (L^{+, Q(u)} \cP)_k \cong L^+ \cP_{\fa^{\flat}} \text{ and } (L^{+, Q(u)} \cP)_{F} \cong L^+ \Res_{K/F} G
$$
compatible with the isomorphisms in Proposition \ref{fibers1}.
\end{prop}  
\begin{proof} 
For any $k$-algebra $R$, we have 
$$
L^{+, Q(u)} \cG (R) = \cG_{k[\![u]\!]} \left(R[\![u]\!] \right)
$$
where $\cG_{k[\![u]\!]}$ is a hyperspecial maximal parahoric for the group $G^{\flat} = \cG_{k(\!(u)\!)}$.  For $g \in L^{+, Q(u)} \cG (R)$ let $\overline{g} \in \cG(R)$ be the reduction modulo $u$.   Then $g$ lies in $L^{+, Q(u)} \cP (R)$ if and only if $\overline{g} P_k \overline{g}^{-1} = P_k$ where $P_k$ is the parabolic of $\cG_k$ defining $\fa^{\flat}$.  Thus, $\overline{g} \in N_{\cG_k}(P_k) = P_k$ and so $g \in L^+ \cP_{\fa^{\flat}}(R)$.

On the generic fiber, we have $(L^{+, Q(u)} \cP)_F \cong (L^{+, Q(u)} \cG)_F$.   Furthermore, 
$$
L^{+, Q(u)} \cG(R) = G \left(\widehat{R}_{Q(u)} \right)
$$
for any $F$-algebra $R$. As in the proof of Proposition \ref{fibers1}, we work over a splitting field $L$ of $Q(u)$ so $Q(u) = \prod_{i=1}^d (u - \varpi_i)$ in $L$.  For any $L$-algebra $R$, we have 
$$
G \left(\widehat{R}_{Q(u)} \right) \cong \prod_i G \left(R[\![u - \varpi_i]\!] \right) = \prod_{\psi:K \ra L} L^+ G(R)
$$
and we conclude by Galois descent. 
\end{proof}

\subsection{Local models}

Fix a geometric cocharacter $\mu$ of $\Res_{K/F} G$.  Let $E$ denote the reflex field of the conjugacy class $\{\mu\}$.

\begin{defn} \label{defnlocmodel} Let $S_{\Res_{K/F} G}(\mu) \subset (\Gr_{\Res_{K/F} G})_E$ be the closed affine Schubert variety associated to $\{\mu\}$, and let $1_{\cG_F/P_F}$ denote the closed point of $\cG_F/P_F$ corresponding to $P_F$. Then the \emph{local model} $M_{P}(\mu)$ associated to $\{\mu\}$ is the Zariski closure of $1_{\cG_F/P_F} \times S_{\Res_{K/F} G}(\mu) $ in $\left( \Fl^{Q(u)}_{\cG} \right)_{\cO_E}$. It is a flat projective scheme over $\Spec \cO_E$.
\end{defn}

\begin{rmk} In \S 4, we change notation and denote our local models by $M_{\cG}(\mu)$ to indicate a dependence on a certain ``parahoric" group scheme $\cG$ which is different from the $\cG$ appearing in this section. Both constructions produce the same local models by Proposition \ref{comp2}.
\end{rmk}

The main theorem on the geometry of local models is:
\begin{thm} \label{locmain1} Suppose that $p \nmid |\pi_1(G_{\der})|$.  Then the scheme $M_{P}(\mu)$ is normal. In addition, the special fiber $\overline{M}_{P}(\mu)$ is reduced, and each geometric irreducible component of $\overline{M}_{P}(\mu)$ is normal, Cohen-Macauley, and Frobenius split.
\end{thm}

As in \cite{PZ}, Theorem \ref{locmain} will follow from identifying the special fiber of $M_{P}(\mu)$ with a union of affine Schubert varieties in $\Fl_{\fa^{\flat}}$ indexed by the so-called $\la_{\mu}$-admissible set.  We briefly recall the definition.  For more detail, see \cite[\S 9.1]{PZ} or \cite[\S 4.3]{PRS}.  Let $\widetilde{W}^{\flat}$ be the Iwahori-Weyl group of the split group $\cG_{\overline{k}(\!(u)\!)}$ which sits in an exact sequence
$$
0 \ra X_*(T^{\flat}) \ra \widetilde{W}^{\flat} \ra  W \ra 0
$$
where $T^{\flat}$ is any maximal torus  of $\cG_{k(\!(u)\!)}$ and $W$ is the absolute Weyl group of $(\cG_{k(\!(u)\!)}, T^{\flat})$.  

\begin{defn} \label{admiss1} Let $\la \in X_*(T^{\flat})$. An element $\widetilde{w} \in \widetilde{W}^{\flat}$ is $\la$-admissible if $\widetilde{w} \leq t_{w \la}$ for some $w \in W$ where $t_{w \la}$ is translation by $w \cdot \la$. The set of $\la$-admissible elements is denoted $\Adm(\la)$. If $W_{\fa^{\flat}} \subset W$ is the subgroup associated  to $\fa^{\flat}$, then define
$$
\Adm^{\fa^{\flat}}(\la) = W_{\fa} \Adm(\la) W_{\fa} \subset \widetilde{W}^{\flat}.
$$
\end{defn}    

The generic fiber of $M_{P}(\mu)$ is stable under $L^+ \Res_{K/F} G$ and so by flatness considerations  $M_{P}(\mu)$ is stable under $L^{+, Q(u)} \cP$.  In particular, set-theoretically the special fiber is union of orbits for the action of $L^+ \cP_{\fa^{\flat}}$ on $\Fl_{\fa^{\flat}}$, i.e., affine Schubert varieties.  The theorem below gives a scheme-theoretic description of the special fibers in terms of affine Schubert varieties. 

Let $\cT$ denote a fiberwise maximal torus of $\cG$ contained in $\cP$ which specializes (under $u \mapsto \varpi$) to a maximal torus $T \subset G$. Let  $T^{\flat} := \cT_{k(\!(u)\!)}$ and let $\breve{F}$ be maximal unramified extension of $F$ in $\overline{F}$ with ring of integers $\breve{\cO}$.  Then we can identify
$$
X_*(T) = X_*(T^{\flat})
$$
as the cocharacters of $\cT$ over $\Spec \breve{\cO}[u]$ where $\cT$ is split.  Choose a Borel subgroup $B \subset G$.  There exists $\mu' \in \{ \mu \}$ such that $\mu'$ is valued in $(\Res_{K/F} T)_{\overline{F}}$ and such that $\mu' = (\mu'_{\psi})_{\psi:K \ra \overline{F}}$ with $\mu'_{\psi} \in X_*(T)$ $B$-dominant.  Define
$$
\la_{\mu} = \sum_{\psi:K \ra F} \mu'_{\psi} \in X_*(T) = X_*(T^{\flat}).
$$    
Note that the conjugacy class of $\la_{\mu}$ does not depend on the choice of $B$. 

\begin{thm} \label{thmadm0} Suppose that $p \nmid |\pi_1(G^{\mathrm{der}})|$ . Then
$$
M_{P}(\mu) \otimes_{\cO_E} \overline{k} = \bigcup_{\widetilde{w} \in \Adm^{\fa^{\flat}}(\la_{\mu})} S_{\widetilde{w}}^{\fa^{\flat}}.
$$
\end{thm}  

Theorem \ref{locmain1} follows from Theorem \ref{thmadm0} . The fact that $M_{P}(\mu)$ is normal with reduced special fiber  follows from Theorem \ref{thmadm0} (cf. \cite[pg. 221]{PZ} or discussion after Theorem \ref{locmain}).                                                  

\begin{rmk} He \cite{HeCM} shows that $\bigcup_{\widetilde{w} \in \Adm^{\fa^{\flat}}(\la)} S_{\widetilde{w}}^{\fa^{\flat}}$ is Cohen-Macauley when $\la$ is minuscule and $G^{\flat}$ is unramified. Thus, if $\la_{\mu}$ is minuscule (which is usually not the case) then $M_{P}(\mu)$ is in fact Cohen-Macauley.  Otherwise, it is open question whether the whole local model is Cohen-Macauley.
\end{rmk} 

In this section, we will show the inclusion $\bigcup_{\widetilde{w} \in \Adm^{\fa^{\flat}}(\la_{\mu})} S_{\widetilde{w}}^{\fa^{\flat}} \subset \overline{M}_{P}(\mu)$.  The reverse inclusion is an application of the coherence conjecture of Pappas and Rapoport combined with a product formula (Proposition \ref{productformula}).  The argument is essentially the same as in the ramified case and so we leave it until \S 4.3.  

The extremal elements in $\Adm(\la_{\mu})$ under the Bruhat order are given by $t_{\la'}$ where $\la'$ is in the Weyl group orbit of $\la_{\mu}$. The cocharacter $\la'$ defines a $\overline{k}$-point $\overline{s}_{\la'}$ of $\Fl_{\fa^{\flat}}$ whose orbit closure is $S_{t_{\la'}}^{\fa^{\flat}}$.  Since $M_{P}(\mu)$ is stable under $L^{+, Q(u)} \cP$, to prove the inclusion  $\bigcup_{\widetilde{w} \in \Adm^{\fa^{\flat}}(\la_{\mu})} S_{\widetilde{w}}^{\fa^{\flat}} \subset \overline{M}_{P}(\mu)$, it suffices to show that $\overline{s}_{\la'}$ is the reduction of a point of $S_{\Res_{K/F} G}(\mu)$. 

\begin{prop} \label{section1} Let $\widetilde{E}$ be a finite extension of $\breve{F}$ which contains a splitting field for $K$ over $F$. There exists a point
$$
s_{\la'}:\Spec \cO_{\widetilde{E}} \ra M_{P}(\mu)
$$
mapping the closed point to $\overline{s}_{\la'}$. 
\end{prop} 
\begin{proof}
Let $\mu = (\mu_{\psi})_{\psi:K \ra \widetilde{E}}$ with $\mu_{\psi} \in X_*(T) = X_*(T^{\flat})$.   Since $\la'$ is in Weyl group orbit of $\la_{\mu}$, there exists cocharacters $\mu'_{\psi} \in X_*(T^{\flat})$ in the Weyl group orbit of $\mu_{\psi}$ such that $\la' = \sum_{\psi} \mu'_{\psi}$.

Since we have chosen $\cT \subset \cP$, there is a natural map
$$
\Gr_{\cT}^{Q(u)} \ra \Fl_{\cG}^{Q(u)}.
$$
We briefly recall the construction of $s_{\la'} \subset \Gr_{\cT}^{Q(u)}(\cO_{\widetilde{E}})$ which is essentially the same as in \cite[Proposition 10.2.9]{LevinThesis}. Since $\cT$ is a split torus over $\cO_{\widetilde{E}}$, a element of $\Gr_{\cT}^{Q(u)}(\cO_{\widetilde{E}})$ is the same as a homomorphism from $X^*(T)$ to category of line bundles on $X := \bA^1_{\cO_{\widetilde{E}}}$ with trivialization away from divisor $D$ defined by $Q(u)$. Over $\cO_{\widetilde{E}}$, $Q(u) = \prod_{\psi:K \ra \widetilde{E}} (u - \varpi_{\psi})$.  Let $D_{\psi}$ denote the divisor defined by $u - \varpi_{\psi}$.  The point $s_{\la'}$ is given by the map which assigns to any $\chi \in X^*(T)$ the line bundle $\otimes_{\psi} \cO_X(-D_{\psi})^{\langle \chi, \mu_{\psi} \rangle}$ with its canonical trivialization over $X - D$.            
\end{proof} 

We now relate the construction from this section to the more general construction which we carry out in \S 4.   Let $\cG_{\fa}$ be the parahoric group scheme over $\cO[u]$ associated to $\fa$ in \S 3.  In this setting, we can be more explicit about the construction.  Namely, $\cG_{\fa}$ is constructed as the dilation of $\cG := \cG_0 \otimes_{\cO} \cO[u]$ along the closed subscheme $P \subset \cG_0 = (\cG)|_ {u = 0}$.  Let $i_{\fa}$ denote the natural homomorphism $\cG_{\fa} \ra \cG$.  Then $i_{\fa}[1/u]$ is an isomorphism and $i_{\fa}|_{u = 0}$ factors through $P$.

\begin{prop} \label{comp2} There is a closed immersion
$$
\Theta:\Fl_{\cG_{\fa}}^{Q(u)} \ra \Fl_{\cG}^{Q(u)}
$$
which is an isomorphism on special fibers and such that the image of the generic fiber $\left(\Fl_{\cG_{\fa}}^{Q(u)} \right)_F$ is $1_{\cG_F/\cP_F} \times \Gr_{\Res_{K/F} G}$.
\end{prop} 
\begin{proof} 
First, we construct $\Theta$.  Let $(\cE', \beta') \in \Fl_{\cG_{\fa}}^{Q(u)}(R)$.  Consider the $\cG$-bundle $\cE = (i_{\fa})_*(\cE')$ with trivialization $\beta = (i_{\fa})_*(\beta')$.   Define $\Theta(\cE', \beta') = (\cE, \beta, \eps)$ where $\eps$ is the reduction to $P$ of $\cE|_{u = 0}$ induced by the factorization
$$
(\cG_{\fa})|_{u = 0} \ra P \ra \cG.
$$

The fact that $\Theta_k$ is an isomorphism follows from comparing the two different descriptions of the affine flag variety.  On the left, we have a moduli space of $\cP_{\fa^{\flat}}$-bundles with trivialization and on the right the moduli of triples $(\cE, \beta, \eps)$ as in \cite{Gaitsgory}. The equivalence of these two descriptions is implicit in for example \cite{PRTwisted}.  It can be deduced from the fact that both are isomorphic to the fpqc quotient $L G^{\flat} / L^+ \cP_{\fa^{\flat}}$ (see \cite[\S 0]{PRTwisted}).

On the generic fiber, we consider the local description as in Definitions \ref{bigflagloc0} or \ref{bigflagloc} so we see that $\left( \Fl_{\cG_{\fa}}^{Q(u)} \right)_F$ only depends on the completion of $\cG_{\fa}$ along $Q(u)$.  Furthermore, $i_{\fa}$ is an isomorphism when restricted to completion of $F[u]$ along $Q(u)$. Thus, 
$$
\left(\Fl_{\cG_{\fa}}^{Q(u)} \right)_F \cong \left(\Gr_{\cG}^{Q(u)}\right)_F \cong \Gr_{\Res_{K/F} G}.
$$
It is easy to check that for any $F$-algebra $R$ and any $(\cE', \beta') \in \Fl_{\cG_{\fa}}^{Q(u)}(R)$ the reduction to $\cP$ in $\Theta(\cE', \beta')$ is the trivial one. Since both sides are ind-proper, $\Theta$ is a closed immersion.  
\end{proof}

\section{Parahoric group schemes}

In this section, we construct the ``Bruhat-Tits" group schemes $\cG$ over the two-dimensional base $\Spec \cO_{K_0}[u]$ which will be used in the  construction of local models in \S 4. The group $\cG$ extends $G/K$ in the sense that $\cG_{u \mapsto \varpi, K}$ is isomorphic to $G$. Furthermore, $\cG$ will be reductive when restricted to $\Spec \cO_{K_0}[u, u^{-1}]$.  The specialization $\cG_{u \mapsto \varpi, \cO_K}$ will be a parahoric group scheme which is given as input into the construction.  The completion at $u$ on the special fiber $\cG_{k_0[\![u]\!]}$ will also be a parahoric group scheme.  Although the argument is the same, the key difference in our construction from \cite{PZ} is that given $G/K$ (a tame group) we construct $\cG$ over $\cO_{K_0}[u]$ as opposed to over $\cO_K[u]$. This is important so that in \S 4 we can define a moduli problem over $\cO_F$.        

Let $K/F$ be a finite extension and choose a uniformizer $\varpi$ of $K$. Let $G$ be a connected reductive group over $K$ which splits over a tamely ramified extension.  Choose a tame extension $\widetilde{K}/K$ which splits $G$.   Let $\widetilde{K}_0$ denote the maximal unramified over $K$ subextension of $\widetilde{K}_0$.  After possibly enlarging $\widetilde{K}$, we can assume that
\begin{itemize}
\item $G_{\widetilde{K}_0}$ is quasi-split;
\item $\widetilde{K}/K$ is Galois;
\item there is uniformizer $\widetilde{\varpi}$ of $\widetilde{K}$ such that $\widetilde{\varpi}^{\widetilde{e}} = \varpi$
\end{itemize}
as in \S 2.a of \cite{PZ}.  Set $\widetilde{e} := [\widetilde{K}:\widetilde{K}_0]$. If $\widetilde{k}$ is the residue field of $\widetilde{K}$, then let $\widetilde{\cO}_0 = W \left(\widetilde{k} \right) \otimes_{W(k_0)} \cO_{K_0}$.

\subsection{Reductive group schemes over $\cO_{K_0}[u, u^{-1}]$}

Let $K_0 \subset K$ be the maximal unramified over $F$ subextension of $K$. We will now construct a group scheme $\underline{G}$ over $\cO_{K_0}[u, u^{-1}]$ which extends $G$ in the sense that its base change
$$
\underline{G} \otimes_{\cO_{K_0}[u, u^{-1}]} K, \quad u \mapsto \varpi
$$
is isomorphic to $G$. The construction is essentially the same as in \S 3 of \cite{PZ} except that we realize the tame descent over $\Spec \cO_{K_0}[u, u^{-1}]$ instead of over $\Spec  \cO_{K}[u, u^{-1}]$. We highlight the main points of the argument.
  
Consider the following diagram:
$$
\xymatrix{ \label{tamed}
\Spec \widetilde{\cO}_0[v] \ar[d]^{u \mapsto v^{e}} &  \Spec \widetilde{K}  \ar[l]^{v \mapsto \widetilde{\varpi}} \ar[d] \\
\Spec \cO_{K_0}[u]  & \Spec K \ar[l]^{u \mapsto \varpi}
}
$$
As in \cite[\S 3.2.1]{PZ}, the action of $\Gamma := \Gal(\widetilde{K}/K)$ on $\widetilde{K}$ can be extended to an action on  $\widetilde{\cO}_0[v]$ over $\cO_{K_0}[u]$.  The left vertical map is finite flat of degree $\widetilde{e}$ and \'etale away from $u = 0$. 

\begin{rmk} Another difference between this setup and that of \cite{PZ} is that kernel of the map $\cO_{K_0}[u, u^{-1}] \ra K$ is generated by the minimal polynomial $Q(u)$ of $\varpi$ over $K_0$ as opposed to $u - \varpi$.  
\end{rmk}  

\begin{prop} \label{underlineG} Let $(G, A, S, P)$ be a \emph{rigidification} of $G$ $($see \cite[Definition 1.7]{PZ}$)$. There exists a quadruple $(\underline{G}, \underline{A}, \underline{S}, \underline{P})$ over $\Spec \cO_{K_0}[u, u^{-1}]$  where $\underline{G}$ is a connected reductive group, $ \underline{A}$ is a split torus of $\underline{G}$, $\underline{S}$ is a torus containing $\underline{A}$ which becomes split over $\widetilde{\cO}_0[u, u^{-1}]$, and $\underline{P}$ is a parabolic subgroup scheme containing $\underline{M} := Z_{\underline{G}}(\underline{A})$, which extends 
the rigidification $(G, A, S, P)$.
\end{prop} 

The construction is identical to \S 3.3 of \cite{PZ}.  We recall the key points which will be relevant later.   The first step is to extend a quasi-split form $G^*$ of $G$ to a quasi-split group $\underline{G}^*$ over $\Spec \cO_{K_0}[u,u^{-1}]$.     Let $H$ denote a split form of $G$ over $\cO_{K_0}$.  Then $G$ defines a cohomology class $[c] \in H^1(\Gamma, \Aut(H_{\widetilde{K}}))$.   Let $[c^*]$ denote the image of $[c]$ in $H^1(\Gamma, \mathrm{Out}(H))$ where $\mathrm{Out}(H)$ denotes the outer automorphism group of $H$.   A choice of pinning of $H$ over $\cO_{K_0}$ (and hence over $K$) defines a splitting of the map $\Aut(H_{\widetilde{K}}) \ra \Out(H)$.   Under this splitting, $[c^*]$ defines a pinned quasi-split group $G^*$ over $K$ which becomes isomorphic to $G$ over $\widetilde{K}_0$ (with an appropriate choice of pinning).  

The outer automorphisms extend to automorphisms of the (pinned) split group over $\widetilde{\cO}_0[v, v^{-1}]$; in particular, $[c^*]$ lifts to a class in $H^1(\Gamma, \Aut( H_{\widetilde{\cO}_0[v, v^{-1}]}))$.  This class defines a reductive (pinned) quasi-split group $\underline{G}^*$ over $\Spec \cO_{K_0}[u,u^{-1}]$ which extends $G^*$.    

The difficult part which is carried out in \cite[\S 3.3.2]{PZ} is constructing $\underline{G}$ from $\underline{G}^*$.  This is done by ``unramified" descent from the $\widetilde{\cO}_{0}[u,u^{-1}]$ to $\cO_{K_0}[u, u^{-1}]$. Since $G$ is an unramified inner form of $G^*$, it is given by a class $[c] \in H^1 \left(\widehat{\Z}, G^*_{\ad}(\breve{K}) \right)$ where $\breve{K}$ is the maximal unramified extension of $K$.  If we choose a Frobenius element $\sigma \in \Gal(\breve{K}/K)$, then the cohomology class $[c]$ is represented by a cocycle given by an inner automorphism $\Int(g)$ of $G^*_{\breve{K}}$.   The key observations \cite[2.10]{PZ} and \cite[2.14]{PZ} say that we can choose $[c]$ and $\Int(g)$ in a controlled way.  To state this precisely, we need to introduce some more notation. 

\begin{defn} \label{notations} If $(G, A, S, P)$ is a rigidification of $G$, then set $M := Z_G(A)$, a Levi subgroup of $G$. Set $T := Z_G(S) \subset M$ a maximal torus since $S_{\widetilde{K}_0}$ is a maximal $\widetilde{K}_0$-split torus of a quasi-split group. 
\end{defn}  

The quasi-split form $G^*$ of $G$ contains a corresponding Levi subgroup $M^*$ with maximal torus $T^*$.  

\begin{prop} \label{precisec} Let $M'^*$ denote the Levi subgroup of $G_{\ad}^*$ corresponding to $M^*$.  There exists a unique class $[c^{\mathrm{rig}}]$ valued in $N_{M'^*}(T^*_{\ad})(\breve{K})$ whose image is $[c]$ under
$$
H^1(\widehat{\Z}, N_{M'^*}(T^*_{\ad})(\breve{K})) \ra H^1(\widehat{\Z}, M'^*(\breve{K})) \ra H^1(\widehat{\Z}, G_{\ad}^*(\breve{K})).
$$
Furthermore, we have an isomorphism 
$$
\rho:M^*_{\ad} \xrightarrow{\sim} \prod \Res_{E_i/K} \PGL_{m_i}
$$
such that the image of $\mathrm{Int}(g)$ in $M^*_{\ad}$ is given by $\prod g_i$ where the $g_i = \underline{n}_{r_i}(\omega_i)$ $($see the discussion before \cite[Proposition 2.14]{PZ}$)$.
\end{prop}
\begin{proof} See Propositions 2.10 and 2.14 in \cite{PZ}. 
\end{proof}

\begin{rmk} The precise form of the cocycle in Proposition \ref{precisec} is important not just for the construction of $\underline{G}$.  It also important that $\Int(g)$ acts on $T^*$ through the relative Weyl group of $(G^*, S^*)$. This will allow us to compare apartments of $\underline{G}$ and $\underline{G}^*$ under different specializations.  
\end{rmk}

\begin{proof}[$($Sketch of the$)$ proof of Proposition \ref{underlineG}:]  The groups $M^*$ and $T^*$ extend naturally to subgroups $\underline{M}^*$ and $\underline{T}^*$ of $\underline{G}^*$ and similarly for $\underline{G}_{\ad}^*$.   Let $\breve{\cO}_0$ denote the ring of integers of the maximal unramified over $K_0$ subfield of $\breve{K}$. The goal then is to produce a lift of $[c^{\mathrm{rig}}]$ from Proposition \ref{precisec} to $H^1 \left(\widehat{\Z}, N_{M'^*}(T^*_{\ad})(\breve{\cO}_0[u, u^{-1}]) \right)$. Using the exact sequence \cite[(3.5)]{PZ}
$$
1 \ra \underline{Z}^* \ra \underline{N'}^* \ra \underline{N}^*_{\ad} \ra 1
$$
and some computations with the Brauer group, one is reduced to showing that the composition 
$$
H^1 \left(\widehat{\Z}, \underline{N}^*_{\ad}(\breve{\cO}_0[u, u^{-1}] \right) \ra H^1 \left(\widehat{\Z}, \underline{M}^*_{\ad}(\breve{\cO}_0[u, u^{-1}] \right) \ra H^1 \left(\widehat{\Z},M^*_{\ad} \left(\breve{K} \right) \right)
$$ is surjective.  Exactly as in \cite[\S 3.3.3]{PZ}, the group $M_{\ad}$ is anisotropic and hence isomorphic to a product of Weil restrictions of central division algebras.   The same constructions from \S 5.3.1 and 5.3.3 of \cite{PZ} produce the desired classes in $H^1 \left(\widehat{\Z}, \underline{N}^*_{\ad}(\breve{\cO}_0[u, u^{-1}]) \right)$.  
\end{proof} 

\subsection{Iwahori-Weyl group} Recall that for a split pair $(H, T_H)$ over a complete discretely valued field $\kappa$, the Iwahori-Weyl group is given by $N_G(T_H)(\kappa)/T(\cO_{\kappa})$ which has as a quotient the absolute Weyl group $W_0(H, T_H)$.  This can be generalized to non-split groups as well.   

\begin{defn} \label{IW} Let $\kappa$ be a discretely valued field with perfect residue field.  Let $G'$ be a connected reductive group over $\kappa$ and $S'$ a maximal $\kappa$-split torus.  The \emph{relative Weyl group} $W_0(G', S')$ is given by $N_{G'}(S')(\kappa)/Z_{G'}(S')(\kappa)$. The Iwahori-Weyl group $W(G',S', \kappa)$ is defined to be 
$$
N_{G'}(S')(\kappa)/M_1
$$
where $M_1$ is the unique parahoric subgroup of $Z_{G'}(S')$.
\end{defn}

For $G'$ quasi-split, one can say a bit more about the structure of $W(G',A', \kappa)$. The centralizer of $S'$ is a maximal torus $T' := Z_{G'}(S')$.  The unique parahoric subgroup scheme $\cT'$ of $T'$ is the connected component of the N\`eron lft -model of $T'$ (see \cite[\S 10]{Neron}, \cite[\S 3.b]{PRTwisted}). 

\begin{prop} \label{lftmodel} Assume $T'$ splits over a Galois extension $\tilde{\kappa}$ with $\Gamma := \Gal(\tilde{\kappa}/\kappa)$ of order prime to $p$.  If $\widetilde{\cT}$ is the split model for $T_{\widetilde{\kappa}}$ over $\cO_{\tilde{\kappa}}$ with the induced action of $\Gamma$ then the unique parahoric subgroup scheme $\cT$  is isomorphic to the connected component of $(\Res_{\cO_{\tilde{\kappa}}/\cO_{\kappa}}  \tilde{\cT})^{\Gamma}$.
\end{prop}
\begin{proof}  By \cite[Proposition 3.4]{Edix}, $(\Res_{\cO_{\tilde{\kappa}}/\cO_{\kappa}}  \tilde{\cT})^{\Gamma}$ is a smooth group scheme over $\Spec \cO_{\kappa}$ whose generic fiber is $T$.   By Proposition 4 in \cite{Neron}, the connected component of the identity is the same as that of the N\`eron lft model for $T$.  
\end{proof} 

We now take $\kappa = \overline{k}(\!(u)\!)$ where $\overline{k}$ is algebraic closed.  Take $I := \Gal(\overline{\kappa}/\kappa)$, $\widetilde{W}' := W(G', S', \kappa)$, and $W_0' := W_0(G', S')$     Then there is an exact sequence
\begin{equation} \label{w1}
0 \ra X_*(T')_I \ra \widetilde{W}' \ra W_0' \ra 0
\end{equation}
where $X_*(T)_I$ denotes the coinvariants (see \cite[(9.14)]{PZ}).  More precisely, the Kottwitz homomorphism (see \cite[Lemma 5]{HR}) identifies $\cT(\kappa)/ \cT(\cO_{\kappa})$ with $X_*(T)_I$.  If $\la \in X_*(T)_I$, then $t_{\la}$ will denote translation by $\la$ considered as an element of $\widetilde{W}'$.

When $G'$ is simply-connected $\widetilde{W}'$ is a Coxeter group and so has a length function and a Bruhat order. For any $G'$, $\widetilde{W}'$ has the structure of a quasi-Coxeter group (see \cite[\S 9.1.1]{PZ}) and so the Bruhat order can be extended to $\widetilde{W}'$.  We will denote this partial ordering by $\leq$. 

Bruhat-Tits theory associates to any connected reductive group $G'$ over $\kappa$ the (extended) building $\cB(G', \kappa)$ on which $G'(\kappa)$ acts and to any pair $(G', S')$ an apartment $\cA(G', S', \kappa)$.  For $G'$ simply-connected, the parahoric subgroups  of $G'(\kappa)$ are exactly the stabilizer of the facets of $\cB(G', \kappa)$.   If $\fa'$ is a facet of $\cB(G', \kappa)$, we denote corresponding parahoric by $\cP_{\fa'}$.  There is a unique \emph{parahoric group scheme} over $\cO_{\kappa}$ associated to $\fa'$ which abusing notation we also denote by $\cP_{\fa'}$.  It is a smooth affine group scheme with geometrically connected fibers.       

The Iwahori-Weyl group $\widetilde{W}'$ acts on the apartment $\cA(G', S', \kappa)$.  If $\fa'$ is a facet of $\cA(G', S', \kappa)$, then set 
$$
W_{\fa'}' := (N_{G'}(S')(\kappa) \cap P_{\fa'})/\cT(\cO_{\kappa}) \subset \widetilde{W}'.
$$
If $G'$ is simply-connected, $W_{\fa'}$ is the stabilizer of $\fa'$ in $\widetilde{W}'$.

To any $\mu' \in X_*(T')$, one can associate a $W_0'$-orbit $\Lambda_{\mu'}$ in $X_*(T)_{I}$ (see \S 9.1.2 of \cite{PZ} or the discussion after Corollary 2.8 in \cite{RicharzAG}).  Fix a $\kappa$-rational Borel $B'$ of $G'$ and let $\overline{\mu}'$  be the image in $X_*(T)_{I}$ of the unique $B'$-dominant cocharacter in the orbit of $\mu'$ under the absolute Weyl group of $(G', T')$.  Set $\Lambda_{\mu'} = W_0' \cdot \overline{\mu}'$.  This is independent of the choice of $B'$ and only depend on the Weyl group orbit of $\mu'$. 

We take the opportunity here to define the $\mu'$-admissible set which was introduced by Kottwitz and Rapoport.
 
\begin{defn} \label{admissible} For $\mu' \in X_*(T')$, define the \emph{$\mu'$-admissible} set
$$
\Adm(\mu') = \{ w \in \widetilde{W}' \mid w \leq t_{\la} \text{ for some } \la \in \Lambda_{\mu'} \}.
$$  
If $\fa'$ is a facet of the apartment $\cA(G', S', \kappa)$, define
$$
\Adm^{\fa'}(\mu') = W_{\fa'}' \Adm(\mu') W_{\fa'}'. 
$$
Note that $\Lambda_{\mu'}$ and hence $\Adm(\mu')$ only depends on the geometric conjugacy class of $\mu'$. 
\end{defn} 
 
\subsection{Parahoric group schemes}

The input into the construction of a local model for $\Res_{K/F} G$ includes a parahoric subgroup or equivalently a facet $\fa$ in the Bruhat-Tits building $\cB(\Res_{K/F} G, F)$. There is a natural isomorphism of the (extended) Bruhat-Tits buildings
$$
\cB(\Res_{K/F} G, F) \cong \cB(G, K)
$$
(see, for example, \cite[pg, 172]{PrasadFixedpoints}). Thus, a parahoric subgroup of $(\Res_{K/F} G)(F)$ is the same as a parahoric subgroup of $G(K)$. Fix a facet $\fa$ of $\cB(G, K)$.  

The special fibers of local models naturally live in affine flag varieties which are associated groups over $\F(\!(u)\!)$ where $\F$ is a finite field. We would now like to associate to $G$ and $\fa$ (up to some choices) a pair $(G^{\flat}, \fa^{\flat})$ where $G^{\flat}$ is a connected reductive group over $k_0(\!(u)\!)$ and $\fa^{\flat}$ is a facet of $\cB(G^{\flat}, k_0(\!(u)\!))$.

Choose a maximal $K$-split torus $A \subset G$ such that $\fa$ is in the apartment $\cA(G, A, K)$ in  $\cB(G, K)$.  Also fix a rigidification $(G, A, S, P)$ over $K$ (Definition 1.7 in \cite{PZ}). Let $\underline{G}^*$ be the pinned quasi-split group over $\cO_{K_0}[u, u^{-1}]$ constructed in \S 3.1.  Let $\breve{K}$ be the completion of the maximal unramified extension of $K$ with residue field $\overline{k}_0$. As in $\S 4.1.2$ of \cite{PZ}, one can identify 
\begin{equation} \label{y1}
\cA(\underline{G}^*_{\breve{K}}, \underline{S}_{\breve{K}}^*, \breve{K}) = \cA(\underline{G}^*_{\overline{k}_0(\!(u)\!)}, \underline{S}^*_{\overline{k}_0(\!(u)\!)}, \overline{k}_0(\!(u)\!))
\end{equation}
equivariantly for the action of the unramified Galois group and compatible with identifying the Iwahori-Weyl groups for $(\underline{G}^*_{\overline{k}_0(\!(u)\!)}, \underline{S}^*_{\overline{k}_0(\!(u)\!)})$ and $(\underline{G}^*_{\breve{K}}, \underline{S}_{\breve{K}}^*)$ (Proposition \ref{IWsame}).   

Let $(\underline{G}, \underline{A}, \underline{S}, \underline{P})$ be the rigidified group over $\Spec \cO_{K_0}[u, u^{-1}]$ constructed in Proposition \ref{underlineG} whose specialization at $u \mapsto \varpi$ is $(G, A, S, P)$.  There is another interesting specialization of $\Spec \cO_{K_0}[u, u^{-1}]$ given by base changing to $k_0(\!(u)\!)$.  Define
\begin{equation} \label{Gflat}
(G^{\flat}, A^{\flat}, S^{\flat}, P^{\flat}) := (\underline{G}, \underline{A}, \underline{S}, \underline{P})_{k_0(\!(u)\!)}.
\end{equation}
We would now like to construct a facet $\fa^{\flat}$ of $\cA(G^{\flat}, A^{\flat}, k_0(\!(u)\!))$ associated to $\fa$.    

Let $\sigma$ denote a generator for $\Gal(\breve{K}/K)$ (considered also as a generator for $\Gal(\overline{k_0}/k_0)$). Unramified descent gives
$$
\cA(G, A, K) = \cA(\underline{G}_{\breve{K}}, \underline{S}_{\breve{K}}, \breve{K})^{\sigma} \text{ and } \cA(G^{\flat}, A^{\flat}, k_0(\!(u)\!)) = \cA(G^{\flat}_{\overline{k}_0(\!(u)\!)}, S^{\flat}_{\overline{k}_0(\!(u)\!)}, \overline{k}_0(\!(u)\!))^{\sigma}.
$$
Over $\breve{\cO}_0[u, u^{-1}]$, the pairs $(\underline{G}, \underline{S})$ and $(\underline{G}^*, \underline{S}^*)$ are isomorphic and hence also over both $\breve{K}$ and $\overline{k}_0(\!(u)\!)$.  Under the identification of apartments, the $\sigma$-action becomes the action of $\Int(\underline{g}) \sigma$. As conjugation by $\underline{g}$ is given by an element in the Iwahori-Weyl group, it is consequence of (\ref{y1}) that
\begin{equation*}
\begin{split}
\cA(\underline{G}_{\breve{K}}, \underline{S}_{\breve{K}}, \breve{K})^{\sigma} &= \cA(\underline{G}^*_{\breve{K}}, \underline{S}_{\breve{K}}^*, \breve{K})^{\Int(\underline{g}) \sigma} \\ 
&= \cA(\underline{G}^*_{\overline{k}_0(\!(u)\!)}, \underline{S}^*_{\overline{k}_0(\!(u)\!)}, \overline{k}_0(\!(u)\!))^{\Int(\underline{g}) \sigma} \\ 
&= \cA(G^{\flat}_{\overline{k}_0(\!(u)\!)}, S^{\flat}_{\overline{k}_0(\!(u)\!)}, \overline{k}_0(\!(u)\!))^{\sigma}  
\end{split}
\end{equation*}
Hence, we can identify $\cA(G, A, K) = \cA(G^{\flat}, A^{\flat}, k_0(\!(u)\!))$  through (\ref{y1}).

\begin{defn} \label{xuuu} Let $(\underline{G}, \underline{A}, \underline{S}, \underline{P})$ be as in Proposition \ref{underlineG}. If $\fa$ is a facet of $\cA(G, A, K)$, then define $\fa^{\flat}$ to be the image of $\fa$ under the above identification. We denote by $\cP_{\fa}$ the parahoric group scheme over $\cO_K$ associated to $\fa$ and $\cP_{\fa^{\flat}}$ the parahoric group scheme over $k_0[\![u]\!]$ associated to $\fa^{\flat}$.  
\end{defn}  

The following is the analogue of \cite[Theorem 4.1]{PZ} in our setting: 

\begin{thm} \label{BTscheme} There is a unique, smooth, affine group scheme $\cG \ra \Spec (\cO_{K_0}[u])$ with connected fibers and with the following properties:
\begin{enumerate}
\item The group scheme $\cG$ extends $\underline{G}$ from $\Spec (\cO_{K_0}[u, u^{-1}])$.
\item The base change of $\cG$ along $\Spec \cO_K \ra  \Spec (\cO_{K_0}[u]) $ given by $u \mapsto \varpi$ is the parahoric group scheme $\cP_{\fa}$ for $G$.
\item The base change of $\cG$ along $\Spec k_0[\![u]\!] \ra \bA^1_{\cO_{K_0}}$ given by reduction modulo $m_{K_0}$ followed by $u$-adic completion is the parahoric group scheme $\cP_{\fa^{\flat}}$ for $G^{\flat}$.
\end{enumerate}
\end{thm}
\begin{proof}  We highlight the key points of the construction from \cite[\S 4.2.2]{PZ} since all the arguments carry over to our situation as well. Their construction generalizes the construction of the Bruhat-Tits schemes from \cite{BTII}.  Step (a) handles the case where $G$ is split.  In this case, one constructs $\cG$ over $\cO_{K_0}[u]$ using the theory of schematic root datum as in \cite[\S 3]{BTII} given by $\fa$.  If $\fa$ contains a hyperspecial vertex, one could also follow the argument of \cite{Yu} which constructs $\cG$ using dilations from the split form of $G$ over $\cO_{K_0}[u]$.  This is used to show that the group constructed using schematic root datum is affine. 

Step (b) considers the case when $G$ is quasi-split and splits over the totally tame extension obtained by adjoining an $\widetilde{e}$th root of $\varpi$.  From step (a), one has a smooth affine group scheme $\cH_{\fa}$ over $\Spec \cO_{K_0}[v]$.  One then constructs $\cG$ as the ``connected component" of $\left(\Res_ {\cO_{K_0}[v]/\cO_{K_0}[u]} \cH_{\fa} \right)^{\gamma_0 = 1}$ where $\gamma_0$ generates the inertia subgroup of $\Gal(\Gamma = \widetilde{K}/K)$.  The fixed points form a smooth subgroup scheme since the order of $\gamma_0$ is prime to $p$.

Finally as in \cite{BTII}, the general case (step (c)) is handled by unramified descent from the quasi-split case.  The key point is to show that unramified descent datum $^{*} \sigma$ on $\underline{G}^*$ for $\widetilde{\cO}_0[u, u^{-1}]$ over $\cO_{K_0}[u, u^{-1}]$ which gives $\underline{G}$ as a form of $\underline{G}^*$ extends to the Bruhat-Tits group scheme $\cG^*$ constructed in step (b).  This boils down to an integrality statement for $^{*} \sigma$ on $\underline{G}^*_{\widetilde{K}_0 (\!(u)\!)}$ which follows from the unramified descent argument from \cite[1.2.7]{BTII} using that $\fa$ lies in $\cA(G, A, K)$ and so is fixed by Frobenius.      
\end{proof}

We end this section by mentioning a few other properties of $\cG$ from Theorem \ref{BTscheme}:
\begin{prop} \label{BTtorus}
\begin{enumerate}
\item The group scheme $\cG_{\widetilde{\cO}_0[u, u^{-1}]}$ is quasi-split;
\item The centralizer $\underline{T} := Z_{\underline{G}}(\underline{S})$ is a fiberwise maximal torus;
\item Let $T^{\flat} = \underline{T}_{\widetilde{k}_0(\!(u)\!)}$.  The group scheme $\underline{T}_{\widetilde{\cO}_0[u, u^{-1}]}$  extends to a smooth affine group scheme $\cT$ over $\widetilde{\cO}_0[u]$ mapping to $\cG_{\widetilde{\cO}_0[u, u^{-1}]}$ such that $\cT_{\widetilde{k}_0[\![u]\!]}$ is the connected N\'eron lft-model of $T^{\flat}$.  
\end{enumerate} 
\end{prop}  
\begin{proof} For (1), $\underline{G}$ is isomorphic to $\underline{G}^*$ over $\widetilde{\cO}_0[u, u^{-1}]$ which is quasi-split by construction.  For (2), one can check the condition after \'etale base change to  $\widetilde{\cO}_0[u, u^{-1}]$ where $\underline{T}$ becomes isomorphic to $\underline{T}^*$ which is a tame descent of a maximal split torus of the split group $(H, T_H)$ over $\widetilde{\cO}_0[v, v^{-1}]$.

For (3), we refer to the construction on pg. 178-179 \cite{PZ} in the quasi-split case.  Let $\cH$ be the construction for the split form of $G$.   Then, $\cG$ is constructed from the smooth affine group scheme $\cG' := \left( \Res_{\widetilde{\cO}_{0}[v]/\widetilde{\cO}_{0}[u]} \cH \right)^{\Gamma}$.  If $\cT_{\cH}$ is a split torus of $\cH$, then 
$$
\cT' =  (\Res_{\widetilde{\cO}_{0}[v]/\widetilde{\cO}_{0}[u]} \cT_{\cH})^{\Gamma} \subset \cG'
$$
is a smooth closed group scheme.  If $Z$ is the complement of the connected component of $(\cT')_{u =0}$ which contains the identity section. Then $\cT$ can be defined to be $\cT' - Z$ in the same way that $\cG$ is constructed from $\cG'$ on \cite[pg. 179]{PZ}.  Since the connected component of $(\cT')_{u =0}$ certainly maps to the connected component of $\cG'_{u = 0}$, the natural inclusion $\cT \ra \cG'$ factors through $\cG$. 
\end{proof}

We end this section by discussion some of the combinatorial data associated to $G$ and $G^{\flat}$ which will be implicit in many of the arguments. Let $\underline{T}$ be as above and let $\underline{N} := N_{\underline{G}}( \underline{T})$.  Set $T^{\flat} = \underline{T}_{\overline{k}_0(\!(u)\!)}$ and $N^{\flat} := \underline{N}_{\overline{k}_0(\!(u)\!)}$.

\begin{prop} \label{IWsame} Let $I$ be the inertia subgroup of $\Gamma = \Gal(\widetilde{K}/K)$.  
\begin{enumerate}
\item There is an natural $I$-equivariant isomorphism
$$
X_*(T^{\flat}) \cong X_*(T).
$$
\item Let $\widetilde{W}$ be the Iwahori-Weyl of the pair $(\underline{G}_{\breve{K}}, \underline{S}_{\breve{K}})$, and $\widetilde{W}^{\flat}$ the Iwahori-Weyl group of $\left(G^{\flat}_{\overline{k}_0(\!(u)\!)}, S^{\flat}_{\overline{k}_0(\!(u)\!)}\right)$.  Then there is a natural isomorphism 
$$
\widetilde{W}^{\flat} \cong \widetilde{W}
$$
which is compatible with the induced isomorphism $X_*(T^{\flat})_I \cong X_*(T)_I$ considered as subgroups of $\widetilde{W}^{\flat}$ and $\widetilde{W}$ respectively.   
\end{enumerate}
\end{prop}
\begin{proof} For (1), both $T$ and $T^{\flat}$ are defined by tame descent from a split torus $T_H$.  Under the identification of the two tame Galois groups, descent is given by the same representation of $\Gamma$. 

For (2), one can compare the Iwahori-Weyl groups $\widetilde{W}^{\flat}$ and $\widetilde{W}$ with the Iwahori-Weyl group of the pair $(\underline{G}_{\breve{K}(\!(u)\!)}, \underline{S}_{\breve{K}(\!(u)\!)})$ (for the $u$-adic valuation).  Consider the map
\begin{equation} \label{z7}
\underline{N}(\cO_{\breve{K}}[u, u^{-1}])/ \cT(\cO_{\breve{K}}[u]) \ra \underline{N}(\breve{K}(\!(u)\!))/\cT(\breve{K}[\![u]\!]).
\end{equation}
The left side has specialization maps to both $\widetilde{W}^{\flat}$ and $\widetilde{W}$.  We claim that (\ref{z7}) is an isomorphism as are the two specializations.  The quotient $\underline{T}(\cO_{\breve{K}}[u, u^{-1}])/\cT(\cO_{\breve{K}}[u])$ is naturally isomorphic to $X_*(T_H)_I$ and so we get an isomorphism on affine parts.  

The Weyl group scheme $\underline{W} := \underline{N}/\underline{T}$ is finite \'etale over $\Spec \cO_{\breve{K}} [u, u^{-1}]$ and becomes constant over $\Spec \cO_{\breve{K}} [v, v^{-1}]$ where $v^e = u$.  For the quasi-split pairs $(\underline{G}_{\breve{K}}, \underline{S}_{\breve{K}})$ and $\left(G^{\flat}_{\overline{k}_0(\!(u)\!)}, S^{\flat}_{\overline{k}_0(\!(u)\!)}\right)$, the relative Weyl group is given by the $I$-invariant elements of the absolute Weyl group. There is an $I$-equivariant isomorphism between $\underline{W}(\overline{k}_0(\!(v)\!))$ and $\underline{W}(\widetilde{K})$.  
\end{proof}

\section{Ramified local models}
In this section, we construct local models for the group $\Res_{K/F} G$ and prove the main theorem (Theorem \ref{locmain}) on their geometry.  We first build a deformation of an affine flag variety to mixed characteristic using the Bruhat-Tits group $\cG$ over $\bA^1_{\cO_{K_0}}$ from the previous section and discuss its basic properties.  The local models arise as flat projective closed subschemes of this mixed characteristic flag variety.  The main theorem is deduced from the coherence conjecture of Pappas-Rapoport which was proven in \cite{ZhuCoh}. 

\subsection{Affine flag varieties} 

Let $\cG$ be a smooth affine group scheme over $\Spec (\cO_{K_0}[u])$ with connected fibers, for example, $\cG$ as in Theorem \ref{BTscheme}.  Let $Q(u) \in \cO_{K_0}[u]$ generate the kernel of the map $\cO_{K_0}[u] \ra K$ given by $u \mapsto \varpi$.  

\begin{defn} \label{bigflag2} For any $\cO_{K_0}$-algebra $R$, define
$$
\Fl_{\cG, 0}^{Q(u)}(R) := \{\text{iso-classes of pairs } (\cE, \beta) \},
$$ 
where $\cE$ is a $\cG$-bundle over $R[u]$, $\beta$ is a trivialization of $\cE|_{\Spec (R[u])[1/Q(u)]}$. If $\cG$ is a reductive group scheme, then we write $\Fl_{\cG, 0}^{Q(u)}$ as $\Gr_{\cG, 0}^{Q(u)}$.  
\end{defn}

There is also a "local" version of $\Fl_{\cG, 0}^{Q(u)}$.:

\begin{defn} \label{bigflagloc} For any $\cO_{K_0}$-algebra $R$, let $\widehat{R}_{Q(u)}$ denote the $Q(u)$-adic completion of $R[u]$. Define
$$
\Fl_{\cG, 0}^{Q(u), \mathrm{loc}}(R) := \{\text{iso-classes of pairs } (\cE, \beta) \},
$$ 
where $\cE$ is a $\cG$-bundle over $\widehat{R}_{Q(u)}$, $\beta$ is a trivialization of $\cE|_{\Spec \widehat{R}_{Q(u)}[1/Q(u)]}$.
\end{defn}

\begin{prop} The natural map of functors
$$
\Fl_{\cG, 0}^{Q(u)} \ra \Fl_{\cG, 0}^{Q(u), \mathrm{loc}}
$$
given by completion at $Q(u)$ is an equivalence.  
\end{prop}
\begin{proof} This follows from Lemma 6.1 of \cite{PZ} which generalizes the Beauville-Laszlo descent lemma to the group $\cG$ over the two-dimensional base $\cO_{K_0}[u]$.  There they work with completion at $u - r$ for some $r \in R$ but the proof is the same for completion at $Q(u)$.  
\end{proof}

From now on, we will use the two descriptions given in Definitions \ref{bigflag2} and \ref{bigflagloc} interchangeably and will use $\Fl_{\cG, 0}^{Q(u)}$ to denote either moduli problem. 

\begin{prop} \label{indscheme} The functor $\Fl_{\cG, 0}^{Q(u)}$ is represented by an ind-scheme of ind-finite type over $\Spec \cO_{K_0}$. 
\end{prop}
\begin{proof}By \cite[Corollary 11.7]{PZ}, there exists a closed immersion $i:\cG \ra \GL_n$ such that the fppf quotient $\GL_n/\cG$ is represented by a smooth quasi-affine scheme.  A standard argument as in \cite[Appendix]{Gaitsgory} or \cite[Proposition 10.1.13]{LevinThesis} says that the induced map $i_*:\Fl_{\cG, 0}^{Q(u)} \ra \Gr_{\GL_n, 0}^{Q(u)}$ is a locally closed immersion.  The functor $ \Fl_{\GL_n, 0}^{Q(u)}$ is represented by an ind-scheme of ind-finite type via its description in terms of $Q(u)$-lattices in Example \ref{GLN}.  
\end{proof}

Define a pro-algebraic group over $\Spec \cO_{K_0}$ by 
$$
L^{+, Q(u)} \cG (R) := \varprojlim \cG \left(R[u]/Q(u)^N \right) = \cG \left( \widehat{R}_{Q(u)} \right).
$$
For any $N \geq 1$, the functor $R \mapsto \cG \left(R[u]/Q(u)^N \right)$ is represented by the smooth affine group scheme $\Res_{\left(\cO_{K_0}[u]/Q(u)^N \right)/\cO_{K_0}} \cG$. Using the local description (\ref{bigflagloc}) of $\Fl_{\cG, 0}^{Q(u)}$, we see that $L^{+, Q(u)} \cG$ acts on $\Fl_{\cG, 0}^{Q(u)}$ by changing the trivialization. 

The functor $L^{+, Q(u)} \cG$ is a version of the positive loop group; there is also a version the loop group.  For any $\cO_{K_0}$-algebra $R$, if  $\widehat{R}_{Q(u)}$ denotes the $Q(u)$-adic completion of $R[u]$, then $L^{Q(u)} \cG$ is the functor given by 
\begin{equation} \label{loopgp}
 R \mapsto \cG(\widehat{R}_{Q(u)}[1/Q(u)])
\end{equation}
As in \cite[\S 6.2.4]{ZhuCoh}, $L^{Q(u)} \cG$ is represented by a formally smooth ind-scheme over $\cO_{K_0}$ but we won't use this fact. There is a natural map
\begin{equation} \label{l2}
L^{Q(u)} \cG(R) \ra \Fl_{\cG, 0}^{Q(u)}(R)
\end{equation} 
which assigns to $g \in  \cG \left(\widehat{R}_{Q(u)}[1/Q(u)] \right)$ the trivial $\cG$-bundle $\cE^0_{\cG}$ over $\Spec \widehat{R}_{Q(u)}$ with trivialization over $\Spec \widehat{R}_{Q(u)}[1/Q(u)]$  given by $g$.

\begin{prop} There exists finite type closed subschemes $\{Z_i \}$ of  $\Fl_{\cG, 0}^{Q(u)}$ with $\varinjlim_i Z_i = \Fl_{\cG, 0}^{Q(u)}$ such that each $Z_i$ is stable under the action  of $L^{+, Q(u)} \cG$ and the action of $L^{+, Q(u)} \cG$ on $Z_i$ factors through $\Res_{(\cO_{K_0}[u]/Q(u)^N)/\cO} \cG $ for some $N \gg 0$ $($i.e., the action is nice in the sense of $\cite[\S A.3]{Gaitsgory})$. 
\end{prop}
\begin{proof} Using the locally closed immersion $i_*:\Fl^{Q(u)}_{\cG, 0} \ia \Gr^{Q(u)}_{\GL_n, 0}$ from the proof of Proposition \ref{indscheme}, it suffices to show that the action of $L^{+, Q(u)} \GL_n$ on $\Fl_{\GL_n, 0}^{Q(u)} = \Gr_{\GL_n}^{Q(u)}$ is nice.  Let $M_0 = (\cO_{K_0}[u])^n$ from Example \ref{GLN}. As in \cite[Theorem 10.1.17]{LevinThesis}, we have 
$$
 \Gr_{\GL_n}^{Q(u)} = \varinjlim X^{Q(u), N}
$$
where $X^{Q(u), N}(R)$ is the set of $R[u]$-module quotients of $$(Q(u)^{-N} (M_0 \otimes_{\cO_{K_0}} R))/(Q(u)^{N} (M_0 \otimes_{\cO_{K_0}} R))$$ which are $R$-projective.   Each $X^{Q(u), N}$ is $L^{+, Q(u)} \GL_n$ stable and the action is through a finite type quotient.   
\end{proof} 

\begin{prop} \label{genfiber} The generic fiber $(L^{+, Q(u)} \cG)_{K_0}$ is isomorphic to the positive loop group $L^+ \Res_{K/K_0} G$.  Furthermore, there is an isomorphism
$$
\left(\Fl_{\cG, 0}^{Q(u)}\right)_{K_0} \cong \Gr_{\Res_{K/K_0} G}
$$
identifying the action of $(L^{+, Q(u)} \cG)_{K_0}$ with the action of  $L^+ \Res_{K/K_0} G$.
\end{prop}
\begin{proof} One can give an algebraic proof as in \cite[Proposition 10.1.6]{LevinThesis}.  However, we give a moduli-theoretic description of the isomorphism.  We construct the isomorphism over a splitting field $E$ for $Q(u)$ and then conclude by Galois descent from $L$ to $K_0$.

For each $K_0$-embedding $\psi:K \ra L$, let $D_{\psi}$ be the divisor on $\bA^1_{L}$ defined by $u - \psi(\varpi)$.  Set $D := \cup_{\psi} D_{\psi}$.   If $(\cE, \beta)$ is a $\cG$-bundle trivialized away from $D$, we can construct for each $\psi$ a $\cG$-bundles $\cE_{\psi}$ together with a trivialization $\beta_{\psi}$ on the complement of $D_{\psi}$ by gluing the trivial bundle on $\bA^1_L  - D_{\psi}$ and $\cE|_{\bA^1_E - \cup_{\psi' \neq \psi} D_{\psi'}}$ via the trivialization $\beta$ of $\cE$ on the intersection.    In fact, this process defines an equivalence of categories between pairs $(\cE, \beta)$ and tuples $\{ (\cE_{\psi}, \beta_{\psi}) \}_{\psi}$.  The inverse is given by gluing the $\cE_{\psi}|_{\bA^1_L - \cup_{\psi' \neq \psi} D_{\psi'}}$ along $\bA^1_L - D$ with gluing data given by $\{ \beta_{\psi'}^{-1} \circ \beta_{\psi} \}$. This is standard argument used in the study of BD/convolution Grassmanians (see, for example,  \cite[Proposition 5]{Gaitsgory}). 

Let $L[\![u - \varpi_{\psi}]\!]$ denote the completion of $\bA^1_L$ along $D_{\psi}$ and let $\cG_{\psi}$ denote the base change of $\cG$ to $L[\![u - \varpi_{\psi}]\!]$.  If $\Gr_{\cG_{\psi}}$ is the (twisted) affine Grassmanian of $\cG_{\psi}$, the equivalence induces an isomorphism
$$
\left(\Fl_{\cG, 0}^{Q(u)}\right)_{L} \cong \bigoplus_{\psi:K \ra L} \Gr_{\cG_{\psi}}.
$$
Using the behavior of bundles under Weil restriction, we have an isomorphism $ \Gr_{\Res_{K/K_0} G} \cong \Res_{K/K_0} \Gr_{G_K}$ (see \cite[\S 2.6]{LevinThesis}).  

It suffices then to show that $\cG_{\psi} \cong G \otimes_{K, \psi} L[\![u - \varpi_{\psi}]\!]$.  Since both $G$ and $\cG_{\psi}$ are defined by descent datum for the group $\Gamma = \Gal(\widetilde{K}/K)$,  it suffices to give a $\Gamma$-equivariant $L[\![u- \varpi_{\psi}]\!]$-algebra isomorphism
\begin{equation} \label{e7}
\widetilde{\cO}_0[v] \otimes_{K_0} L[\![u - \varpi_{\psi}]\!] \cong \widetilde{K} \otimes_{K, \psi} L[\![u - \varpi_{\psi}]\!].
\end{equation}
The left side is naturally isomorphic to the $(v^{\widetilde{e}}- \varpi_{\psi})$-adic completion of $\left(\widetilde{K_0} \otimes_{K_0} L \right) [v]$. The isomorphism (as in \cite[\S 6.2.6]{PZ}) is given by
$$
v \mapsto \left(\widetilde{\varpi} \otimes \left(1 + \frac{(u - \varpi_{\psi})}{\varpi_{\psi}}\right)^{1/\widetilde{e}}\right).
$$
One can check that this map is $\Gamma$-equivariant.  This gives the desired isomorphism $\cG_{\psi} \cong G \otimes_{K, \psi} L[\![u - \psi(\varpi)]\!]$ which induces an isomorphism
$$
\left(\Fl_{\cG, 0}^{Q(u)} \right)_{L} \cong \bigoplus_{\psi:K \ra L} \Gr_{(G \otimes_{K, \psi} L)}.
$$
Similarly, the loop group $\left(L^{+, Q(u)} \cG \right)_{K_0}$ over $L$ is identified with the completion of $\cG$ along $D$ which decomposes as a product of the completions of $\cG_{\psi}$ along each $D_{\psi}$.  
\end{proof}

\begin{rmk} In the proof of Proposition \ref{genfiber}, we chose an $\widetilde{e}$th root of $\left(1 + \frac{(u - \psi(\varpi))}{\psi(\varpi)}\right)$ in $L[\![u - \varpi]\!]$. To be consistent with calculations in the \S 4.2, we fix the choice to be the $\widetilde{e}$th root whose constant term is 1.
\end{rmk}   

\begin{prop} \label{specfiber}  The special fiber  $\left( L^{+, Q(u)} \cG \right)_{k_0}$ is isomorphic to $L^+ \cP_{\fa^ {\flat}}$, where $\cP_{\fa^{\flat}}$ is the parahoric group scheme over $\Spec k_0[\![u]\!]$ associated to the facet $\fa^{\flat}$ defined in \ref{xuuu}.  Furthermore,  there is an isomorphism 
$$
\left(\Fl_{\cG, 0}^{Q(u)} \right)_{k_0} \cong \Fl_{\cP_{\fa^{\flat}}}
$$
identifying the action of $\left(L^{+, Q(u)} \cG \right)_{k_0}$ with the action of $L^+ \cP_{\fa^{\flat}}$.  
\end{prop}
\begin{proof} The special fiber only depends on $\cG$ over $k_0[\![u]\!]$ which by Theorem \ref{BTscheme} is isomorphic to $ \cP_{\fa^{\flat}}$.  Since $Q(u) \equiv u^{[K:K_0]}$ over $k_0$, we have for any $k_0$-algebra $R$
$$
L^{+, Q(u)} \cG (R) = \cP_{\fa^{\flat}}(R[\![u]\!]) = L^+ \cP_{\fa^{\flat}}(R).
$$
The description of $\left(\Fl_{\cG, 0}^{Q(u)}\right)_{k_0}$ from Definition \ref{bigflagloc} is exactly the same as the affine flag variety $\Fl_{\cP_{\fa^{\flat}}}$ (see \S 2.1).
\end{proof} 

As a consequence of Propositions \ref{genfiber} and \ref{specfiber}, both the special and generic fiber of $\Fl_{\cG, 0}^{Q(u)}$ are ind-projective.    We will show later that $\Fl_{\cG, 0}^{Q(u)}$ is in fact ind-projective (see Theorem \ref{AGproj}).  A local model for $\Res_{K/F} G$ should be a flat projective scheme over the ring of integers $\cO$ of $F$ (or, more generally, the reflex field over $F$).  However, the unramified extension $K_0/F$ is harmless which is why we have ignored it thus far and chosen to work over $\cO_{K_0}$.  Define
$$
\Fl_{\cG}^{Q(u)} := \Res_{\cO_{K_0}/\cO} \Fl_{\cG, 0}^{Q(u)} \text{ and } L^{+, Q(u)} \cG := \Res_{\cO_{K_0}/\cO} L^{+, Q(u)}_0 \cG .
$$
The following Proposition summarizes the results of this section as applied to $\Fl_{\cG}^{Q(u)}$. 

\begin{prop} The functor $\Fl_{\cG}^{Q(u)}$ is an ind-scheme of finite type over $\Spec \cO$ with an action of $L^{+, Q(u)} \cG$.  The generic fiber is equivariantly isomorphic to $\Gr_{\Res_{K/F} G}$, and the special fiber is equivariantly isomorphic to $\Res_{k_0/k}  \Fl_{\cP_{\fa^{\flat}}}$.  
\end{prop}
\begin{proof} This is an immediate consequence of Proposition \ref{genfiber} and \ref{specfiber} together with the behavior of affine Grassmanians under Weil restriction \cite[ \S 2.6]{LevinThesis}.  
\end{proof}

\subsection{Local models}

We are now ready to define local models for $\Res_{K/F} G$.  Let $(\Res_{K/F} G, \cP, \{ \mu \})$ be a triple where $\cP$ is a parahoric subgroup of $\Res_{K/F} G$ and $\{ \mu \}$ is a geometric conjugacy class of cocharacters.  Let $\fa$ be the facet of $\cB(G, K)$ associated to $\cP$ as discussed in the beginning of \S 3.3. Choosing a rigidification $(G, A, S, P)$ of $G$ such that $\fa$ is in the apartment $\cA(G, A, K)$,  we constructed in Proposition \ref{underlineG} an extension $\underline{G}$ of $G$ to $\Spec \cO_{K_0}[u, u^{-1}]$.  

Specializing $\underline{G}$ at $k_0(\!(u)\!)$ gave rise to the pair $(G^{\flat}, A^{\flat})$ as well as a facet $\fa^{\flat}$ of the apartment $\cA(G^{\flat}, A^{\flat}, k_0(\!(u)\!))$. Finally, let $\cG$ denote the Bruhat-Tits group scheme over $\cO_{K_0}[u]$ from Theorem \ref{BTscheme}. The conjugacy class $\{ \mu \}$ has a minimal field of definition $E$ of $\overline{F}$ called the \emph{reflex field}.  The conjugacy class $\{ \mu \}$ defines an affine Schubert variety $S_{\Res_{K/F} G}(\mu) \subset (\Gr_ {\Res_{K/F} G})_{\overline{F}}$ which is  defined over $E$. Let $\cO_E$ be the ring of integers of $E$.   

\begin{defn} \label{locmodel}  The \emph{local model} $M_{\cG}(\mu)$ associated to the triple $(\Res_{K/F} G, \cP, \{ \mu \})$ is the closure of 
$$
S_{\Res_{K/F} G}(\mu) \subset   (\Gr_ {\Res_{K/F} G})_{E} \cong \left(\Fl_{\cG}^{Q(u)}\right)_E
$$
in $\left(\Fl_{\cG}^{Q(u)}\right)_{\cO_E}$.  
\end{defn}

We now make a few remarks about $M_{\cG}(\mu)$:
\begin{rmk} In \S 2, we denoted the local model by $M_P(\mu)$ to indicate a dependence on a parabolic subgroup (determined by the parahoric).  In this case, the parahoric $\cP$ determines the group scheme $\cG$ and so we indicate the dependence on $\cG$ with a subscript.  The two constructions agree under the hypotheses of \S 2 by Proposition \ref{comp2}.  
\end{rmk}   

\begin{rmk} \label{isogeny} In the introduction, we stated that we construct local models for any triple $(G', \cP', \mu')$ as long as $p \geq 5$.  So far, we have only dealt with groups $G'$ of the form $\Res_{K/F} G$ where $G$ is tamely ramified.  For $p \geq 5$, the simply-connected cover $\widetilde{G}'$ of $G'$ is always of a product of groups of this form.  There is a subtlety in that $\mu'$ may not lift to $\widetilde{G'}$.  In this case, one can translate $S_{G'}(\mu)_E \subset (\Gr_{G'})_E$ into $(\Gr_{\widetilde{G}})_E$ which is the neutral connected component of $\Gr_{G'}$ and define the local model $M_{\cG}(\mu)$ to be the closure of the translation of $S(\mu)_E$ in $\Fl_{\widetilde{\cG}}^{Q(u)}$. The special fiber of $M_{\cG}(\mu)$ will then satisfy Theorem \ref{locmain}. This strategy can also be used to define alternative local models when $p \mid \pi_1(G^{\der})$ but we will not pursue that here.  
\end{rmk} 

\begin{prop} \label{PZcomp} When $K/F$ is tamely ramified, the local model $M_{\cG}(\mu)$ is isomorphic to the local model constructed in \cite{PZ}.  
\end{prop}
\begin{proof} For simplicity, assume $K/F$ is totally tame of degree $e$.  After possibly making an unramified base change, we can choose a uniformizer $\varpi_K$ of $K$  such that $\varpi_K^e = \varpi_F$ is a uniformizer of $F$.   Consider the map $\Psi:\bA^1_{\cO_F} \ra \bA^1_{\cO_F}$ given by $u \mapsto Q(u) + \varpi_F$.  Define $\cG'$ be the Weil restriction $\Res_{\Psi} \cG$ along the finite flat map $\Psi$.   Using the behavior of bundles under Weil restriction, we have an natural isomorphism
$$
\Fl_{\cG}^{Q(u)} \cong \Gr_{\cG', \varpi_F}
$$
where the right side is the affine Grassmanian associated to $\cG'$ defined in \S 6.2.6 of \cite{PZ}.  It suffices then to check that $\cG'$ is isomorphic to the Bruhat-Tits group scheme in \S 4.3 of \cite{PZ}.  This follows either from looking at the construction carefully or alternatively from checking that $\cG'$ has the characterizing properties given in \cite[Theorem 4.1]{PZ}.  
\end{proof}

\begin{rmk} The local model $M_{\cG}(\mu)$ does not depend on the choice of uniformizer $\varpi$.   The argument is similar to the proof of Proposition \ref{PZcomp} though the details are more tedious.  To compare the constructions for two uniformizers, $\varpi$ and $\varpi'$, one has to compare group schemes obtained after Weil restriction along $\Psi, \Psi':\bA^1_{\cO_F} \ra \bA^1_{\cO_F}$ respectively where $\Psi(u) = Q(u)$ and $\Psi'(u) = Q'(u)$ where $Q(u)$ and $Q'(u)$ are the minimal polynomials of $\varpi$ and $\varpi'$ over $F$. 
\end{rmk}

\begin{prop} \label{Mflat}  The scheme $M_{\cG}(\mu)$ is flat and projective over $\Spec \cO_E$.
\end{prop}
\begin{proof}  We follow the strategy from \cite[\S 1.5]{RicharzAG}. The scheme $M_{\cG}(\mu)$ is flat  over $\Spec \cO_E$ by construction and both fibers are proper.  The generic fiber of $M_{\cG}(\mu)$ is the affine Schubert variety $S_{\Res_{K/F} G}(\mu)$ which is geometrically connected. By Lemma 1.21 of loc. cit. if the special fiber of $M_{\cG}(\mu)$ is non-empty, then $M_{\cG}(\mu)$ is proper.  Non-emptiness of the special fiber is a consequence of Corollary \ref{inclusion} (where we construct special points of $M_{\cG}(\mu) (\overline{k})$). Projectivity follows from properness using the existence of a locally closed immersion $i_*:\Fl_{\cG}^{Q(u)} \ra \Gr^{Q(u)}_{\GL_n}$ as in the proof of Proposition \ref{indscheme} and the fact that $ \Gr^{Q(u)}_{\GL_n}$ is ind-projective.  
\end{proof}

If $k_E$ is the residue field of $E$ then define
$$
\overline{M}_{\cG}(\mu) := M_{\cG}(\mu)_{k_E}
$$
a closed subscheme of the affine flag variety $\Res_{k_0/k}  \Fl_{\cP_{\fa^{\flat}}}$ over $k_E$.   The main theorem of this section is that $\overline{M}_{\cG}(\mu)$ is a union of affine Schubert varieties defined by a certain admissible set (Theorem \ref{thmadm}).  As a consequence, we will see that $\overline{M}_{\cG}(\mu)$ is reduced and normal, and each geometric component is Cohen-Macauley.   

\begin{thm} \label{locmain} Suppose that $p \nmid |\pi_1(G_{\der})|$.  Then the scheme $M_{\cG}(\mu)$ is normal. In addition, the special fiber $\overline{M}_{\cG}(\mu)$ is reduced, and each geometric irreducible component of $\overline{M}_{\cG}(\mu)$ is normal, Cohen-Macauley and Frobenius split.
\end{thm}

We begin with a few reductions.  Clearly, it is enough to check Theorem \ref{locmain} after base changing to  $\cO_{\breve{E}}$ where $\breve{E}$ is the completion of the maximal unramified extension of $E$. Since the generic fiber of $M_{\cG}(\mu)$ is an affine Schubert variety, it is normal (\cite[Theorem 0.3]{PRTwisted}).  As in \cite[Proposition 9.2]{PZ}, one is reduced to showing that $\overline{M}_{\cG}(\mu) \otimes_{k_E} \overline{k}$ is reduced, normal and each irreducible component is Cohen-Macauley.  

Recall that $T$ is maximal torus of $G$.  Fix a representative $\mu \in \{ \mu \}$ which is valued in $(\Res_{K/F} T)_{\overline{F}}$.  For every embedding $\phi:K_0 \subset \breve{E}$, $\mu$ defines a cocharacter $\mu_{\phi} \in X_*(\Res_{K/K_0} T)$.  It is not hard to see that
$$
M_{\cG}(\mu)_{\breve{E}} \cong \prod_{\phi} M_{\cG, 0}(\mu_{\phi}) \otimes_{K_0, \phi} \breve{E}
$$
where $M_{\cG, 0}(\mu_{\psi}) $ are local models for the group $\Res_{K/K_0} G$.   Thus, for the proof of Theorem \ref{locmain}, we can \emph{assume} that $F = K_0$. We make this assumption for the remainder of the section.

We would now like to construct points in the special fiber $\overline{M}_{\cG}(\mu)$. To do this, we consider the subgroup $\cT \subset \cG$ from Proposition \ref{BTtorus} which is a ``Bruhat-Tits group scheme'' for the maximal torus $\underline{T} \subset \underline{G}$ over $\cO_{\breve{E}}[u, u^{-1}]$. Let $T_H = T_{\widetilde{K}}$ be the split torus on which $\Gamma$ acts.  Recall that $\underline{T}_{\cO_{\breve{E}}[v, v^{-1}]} \cong (T_H)_{\cO_{\breve{E}}[v, v^{-1}]}$ where $u \mapsto v^{\widetilde{e}}$.  More specifically, $\cT$ is the ``neutral  connected component'' of
\begin{equation} \label{t2}
\cT' := \Res_{\cO_{\breve{E}}[v]/\cO_{\breve{E}}[u]} (T_H \otimes \cO_{\breve{E}}[v])^{\gamma}
\end{equation}
where $\gamma$ is a generator the inertia subgroup $I \subset \Gamma$ which acts on the split torus $T_H$ and sends $v$ to $\zeta v$ where $\zeta$ is a primitive $\widetilde{e}$th root of unity.  .  

Recall the Kottwitz homomorphism 
$$
\kappa_{\cT}:\cT(\overline{k}(\!(u)\!)) \ra X_*(T_H)_I
$$
discussed in more detail in \S 3.2.  Implicitly, we are identifying $X_*(\underline{T}_{k(\!(u)\!)})$ with $X_*(T_H)$ equivariantly for the action of $I$.  A cocharacter $\mu' \in X_*(\Res_{K/F} T)$ is equivalent to a tuple $(\mu_{\psi}') \in \prod_{\psi:K \ra \overline{F}} X_*(T_{\psi})$.   We would like to take the ``sum'' of the $\mu_{\psi}'$. However, they do not lie canonically in the same group so for each $\psi:K \ra \overline{F}$, choose an embedding $\widetilde{\psi}:\widetilde{K} \ra \overline{F}$ which extends $\psi$.  We have an isomorphism
$$
(T_H) \otimes_{\widetilde{K}, \widetilde{\psi}} \overline{F} \cong T_{\psi}   
$$
and so we can think of $\mu_{\psi}'$ as a cocharacter of $T_H$. Let $\overline{\mu}_{\psi}'$ be the image of $\mu_{\psi}'$ in $X_*(T_H)_I$. The image $\overline{\mu}_{\psi}'$ is independent of the choice $\widetilde{\psi}$. Define
$$
\overline{\la}_{\mu'} := \sum_{\psi} \overline{\mu}_{\psi}' \in X_*(T_H)_I.
$$
The following generalizes \cite[Lemma 9.8]{PZ}.  The proof is somewhat technical and is considerably simpler when the group $G$ is unramified.  The reader who is only interested in unramified groups should refer to Proposition \ref{section1} instead. 
 
\begin{prop} \label{section2} Let $\widetilde{E} \subset \overline{F}$ be a finite extension of $F$ which contains a Galois closure of $\widetilde{K}$ over $F$.  Choose $\mu' \in X_*(\Res_{K/F} T)$.  Then there exists a morphism
$$
s_{\mu'}:\Spec \cO_{\widetilde{E}} \ra L^{Q(u)} \cT
$$
such that 
\begin{enumerate}
\item $(s_{\mu'})_{\widetilde{E}} \in (L^{Q(u)} \cT)(\widetilde{E}) = (L \Res_{K/F} T)(\widetilde{E})$ lies in the same $L^+  \Res_{K/F} T$ orbit as the point induced by $\mu':(\Gm)_{\widetilde{E}} \ra (\Res_{K/F} T)_{\widetilde{E}}$;
\item $\kappa_{\cT}((s_{\mu'})_{\overline{k}}) \equiv \overline{\la}_{\mu'}$ in $X_*(T_H)_I$. 
\end{enumerate}
\end{prop}
\begin{proof} We give a proof based on \cite[Proposition 3.4]{ZhuCoh}. As in loc. cit.,  the split case is quite a bit simpler (see Proposition \ref{section1}).  We first construct $s_{\mu'}$ over $\cO_{\widetilde{E}}$ valued in $\cT'$ (see (\ref{t2})). 

Let $\gamma$ denote a generator for the inertia group $I$ as above. Fix an embedding $\widetilde{\cO}_0 \ra \widetilde{E}$. An $\cO_{\widetilde{E}}$-point of $L^{Q(u)} \cT'$ is the same as a $\gamma$-fixed point in 
$$
T_H \left(\widehat{\cO_{\widetilde{E}}[u]}_{(Q(u))}[1/Q(u)] \otimes_{\widetilde{\cO}_0[u]} \widetilde{\cO}_0[v]\right) = T_H \left(\widehat{\cO_{\widetilde{E}}[v]}_{(Q(u))}[1/Q(u)] \right).
$$  
where we are taking $Q(u)$-adic completions. 

For any embedding $\psi:K \ra \widetilde{E}$, let $\varpi_{\psi}$ denote the image of the uniformizer $\varpi$ of $K$.  Choose $\widetilde{\varpi}_{\psi} \in \widetilde{E}$ such that  $\widetilde{\varpi}_{\psi}^{\widetilde{e}} = \varpi_{\psi}$ (or equivalently $\widetilde{\psi}:\widetilde{K} \ra \widetilde{E}$ extending $\psi$). Define
$$
x_{i, \psi} := \zeta^i v -  \widetilde{\varpi}_{\psi} \in \widehat{\cO_{\widetilde{E}}[v]}_{(Q(u))}[1/Q(u)]
$$
which is a unit since $\prod_{i =1}^{\widetilde{e}} x_{i, \psi} = u - \varpi_{\psi}$. Also, we have
\begin{equation} \label{s5}
\gamma(x_{i, \psi}) = x_{i +1, \psi}
\end{equation}
with $x_{i + \widetilde{e}} = x_i$. 

Let $\mu' = (\mu_{\psi}')_{\psi:K \ra \widetilde{E}}$ where $\mu_{\psi}' \in X_*(T_{\psi})$.  Using the choice of embedding $\widetilde{\psi}:\widetilde{K} \ra \overline{F}$, we think of $\mu_{\psi}' \in X_*(T_H)$. There is a point $s_{\mu'} \in T_H \left (\widehat{\cO_{\widetilde{E}}[v]}_{(Q(u))}[1/Q(u)] \right)$ such that for any character $\chi:T_H \ra \Gm$
$$
\chi(s_{\mu'}) = \prod_{\psi:K \ra \widetilde{E}} \prod_{i=1}^{\widetilde{e}} x_{i, \psi}^{\langle \gamma^i(\chi), \mu'_{\psi} \rangle}.
$$
Since $\gamma$ is acting on $v$ and on $T_H$ (but not on $\widetilde{E}$), $s_{\mu'}$ is $\gamma$-invariant and so defines an element of $L^{Q(u)} \cT'(\cO_{\widetilde{E}})$. An identical argument as in the last paragraph on pg. 20 of \cite{ZhuCoh} with $k(\!(u_1)\!)$ replaced by $\widetilde{E}$ shows that $s_{\mu'}$ factors through the connected component $L^{Q(u)} \cT$.  

Next, we show that $s_{\mu'}$ satisfies properties (1) and (2).  To compute $(s_{\mu'})_{\widetilde{E}}$, we have to unravel the isomorphism $(L^{Q(u)} \cT)_{\widetilde{E}} \cong \prod_{\psi:K \ra \widetilde{E}} L T_{\psi, \widetilde{E}}$.  Explicitly,
\begin{equation}
\begin{split}
(L^{Q(u)} \cT)(\widetilde{E}) = \cT(\widehat{\widetilde{E}[u]}_{(Q(u))}[1/Q(u)]) &= \bigoplus_{\psi} \cT \left(\widetilde{E}(\!(u - \varpi_{\psi})\!) \right) \\
&=  \bigoplus_{\psi} T_H \left(\widehat{\widetilde{E}[v]}_{(u - \varpi_{\psi})}[1/(u - \varpi_{\psi})] \right)^{\gamma =1} \\
&= \bigoplus_{\psi} T_H \left(\left(\widetilde{K} \otimes_{K, \psi} \widetilde{E}\right) (\!(u- \varpi_{\psi})\!) \right)^{\gamma = 1}
\end{split}
\end{equation}
where the last equality is from (\ref{e7}).  Using $\widetilde{\psi}$, we get an isomorphism
$$
T_H \left( \left(\widetilde{K} \otimes_{K, \psi} \widetilde{E} \right)(\!(u- \varpi_{\psi})\!) \right)^{\gamma = 1} \cong T_H \left(\widetilde{E}(\!(u- \varpi_{\psi})\!)\right)
$$
For any two embeddings $\psi$ and $\psi'$, define
$$
z^{\psi}_{i, \psi'} := \zeta^i \left (\widetilde{\varpi}_{\psi} \left( 1 + \frac{(u - \varpi_{\psi})}{\varpi_{\psi}} \right)^{1/\widetilde{e}} \right) - \widetilde{\varpi}_{\psi'}.
$$
Then the $\psi$-component of $\chi(s_{\mu'})_{\widetilde{E}}$ is given by 
$$
\prod_{\psi':K \ra \widetilde{E}} \prod_{i=1}^{\widetilde{e}} (z^{\psi}_{i, \psi'})^{\langle \gamma^i(\chi), \mu'_{\psi'} \rangle} \in \widetilde{E}(\!(u- \varpi_{\psi})\!)^{\times}
$$
As an element of $\widetilde{E}[\![u - \varpi_{\psi}]\!]$, $z^{\psi}_{i, \psi'}$ is a unit when $\psi \neq \psi'$ and when $\psi' = \psi$ but $i \neq \widetilde{e}$ (the second claim uses the choice of $\widetilde{e}$th root).  Furthermore, $z^{\psi}_{\widetilde{e}, \psi}$ vanishes to order 1.   Thus, the $\psi$-component of $\chi(s_{\mu'})_{\widetilde{E}}$ is a unit times $(u- \varpi_{\psi})^{\langle \chi, \mu'_{\psi} \rangle}$ as expected .   

Finally, we consider $s_{\mu'}$ over the residue field $\widetilde{k}$ of $\cO_{\widetilde{E}}$.  Over $\widetilde{k}$, $x_{i, \psi} = x_{i, \psi'}$ for all $\psi':K \ra \widetilde{E}$ so 
$$
\chi(s_{\mu'})_{\overline{k}} =   \prod_{i=1}^{\widetilde{e}} \langle \zeta^i v)^{(\gamma^i(\chi), \la_{\mu'} \rangle}
$$
where $\la_{\mu'} = \sum_{\psi} \mu'_{\psi}$. This is exactly the norm of $\la_{\mu'}(v)$ from $\cT(\widetilde{k}(\!(v)\!))$ to $\cT(\widetilde{k}(\!(u)\!))$ which maps to $\overline{\la}_{\mu'}$ under $\kappa_{\cT}$ using the explicit description of the Kottwitz homomorphism (\cite[\S 7]{KottwitzII}, \cite[pg. 20]{ZhuCoh}).
\end{proof}

 If $w$ is an element of the Iwahori-Weyl group of $G^{\flat}$, then let $S_w^{\fa^{\flat}}$ denote the corresponding locally closed orbit in $\Fl_{\cP_{\fa^{\flat}}}$.  

\begin{cor} \label{inclusion} Let $\la_{\mu'} = \sum_{\psi} \mu_{\psi}'$  where $\mu'_{\psi} \in W \cdot \mu_{\psi}$ where $W$ is the absolute Weyl group of $(G, T)$.  Let $\overline{\la}'$ denote the image of $\la_{\mu'}$ in $X_*(T)_I$.  Then
$$
S^{\fa'}_{t_{\overline{\la}'}} \subset \overline{M}_{\cG}(\mu)
$$
where $t_{\overline{\la}'}$ is the translation element in the Iwahori-Weyl group $\widetilde{W}(G^{\flat}, T^{\flat})$.
\end{cor} 
\begin{proof}
By Proposition \ref{invar}, $\overline{M}_{\cG}(\mu)$ is preserved by the action of $L^+ \cP_{\fa^{\flat}}$ so it suffices to show that a single point of $S^{\fa'}_{t_{\overline{\la}'}}$ is contained in $\overline{M}_{\cG}(\mu)$.  Let $s_{\mu'} \in L^{Q(u)} \cT (\cO_{\widetilde{E}})$ be as in Proposition \ref{section2} applied to $\mu' = (\mu'_{\psi}) \in X_*(\Res_{K/F} T)$.   Since $\mu'$ is conjugate to $\mu$ in $\Res_{K/F} G$, $(s_{\mu'})_{\widetilde{E}} \in S_{\Res_{K/F} G} (\mu)$.  Thus, 
$$
(s_{\mu'})_{\overline{k}} \in \overline{M}_{\cG}(\mu)(\overline{k})
$$
and property (2) of Proposition \ref{section2} guarantees that $(s_{\mu'})_{\overline{k}}$ is also in $S^{\fa'}_{t_{\overline{\la}'}}$. 
\end{proof}

We no longer assume that $K_0 = F$.  
\begin{thm} \label{AGproj} $\Fl_{\cG}^{Q(u)}$ is ind-projective.  
\end{thm}
\begin{proof}  We follow the argument from \cite[\S 1.5]{RicharzAG}. We can reduce immediately to the case of $K_0 = F$.  We can replace $F$ by a finite unramified extension such that $G$ and hence $\cG[1/u]$ becomes quasi-split.  We saw in Proposition \ref{Mflat} that $M_{\cG}(\mu)$ is projective. Let $J$ be the set of $\Gal(\overline{F}/F)$-orbits of conjugacy classes of cocharacters of $\Res_{K/F} G$.  Every $\delta \in J$ defines a closed subscheme $S_{\delta} \subset \Fl_{\cG}^{Q(u)}$ which geometrically is the union of the $S_{\Res_{K/F} G}(\mu)$ for $\mu \in \delta$. Let $M_{\delta}$ be the flat closure of $S_{\delta}$ in $\Fl_{\cG}^{Q(u)}$. The $M_{\delta}$ are projective since all $M_{\cG}(\mu)$ are. The claim then is that
$$
\left(\Fl_{\cG}^{Q(u)} \right)_{\red} = \bigcup_{\delta \in J} M_{\delta}.
$$
The union of the $S_{\delta}$ contain all $\overline{F}$-points of $\Fl_{\cG}^{Q(u)}$ and so cover the reduced generic fiber.  By Corollary \ref{inclusion},  $\bigcup_{\delta \in J} M_{\delta}$ contains the closed $L^+ \cP_{\fa^{\flat}}$-orbits of all translation elements $t_{\la}$ in the Iwahori-Weyl group.  This contains all $\overline{k}$-points of  $\Fl_{\cG}^{Q(u)}$.   This proves the inclusion $\left(\Fl_{\cG}^{Q(u)} \right)_{\red} \subset \bigcup_{\delta \in J} M_{\delta}$.  This other inclusion is clear. 
\end{proof} 

\subsection{Line bundles and the coherence conjecture}

In this subsection, we identify the special fiber of $M_{\cG}(\mu)$ with a union of affine Schubert varieties inside an affine flag variety (Theorem \ref{thmadm}).  The strategy of the proof is the same as in \cite[\S 9]{PZ}.  The inclusion of the affine Schubert varieties in $M_{\cG}(\mu)$ is elementary (though somewhat technical) and follows from Proposition \ref{section2}.  One then constructs an ample line bundle on $M_{\cG}(\mu)$ and shows an equality of Hilbert polynomials.  In \cite{PZ}, the equality of Hilbert polynomials is a deep fact which follows from the coherence conjecture of Pappas and Rapoport proven in \cite{ZhuCoh}.  In our context, we combine the coherence conjecture with the product formula (Proposition \ref{productformula}) to deduce the desired equality.      

\begin{prop} \label{invar} The scheme $M_{\cG}(\mu)$ is invariant under the action of $L^{+, Q(u)} \cG$. In particular, $\overline{M}_{\cG}(\mu)_{\red}$ is a union of affine Schubert varieties of $\Fl_{\cP_{\fa^{\flat}}}$.   
\end{prop}
\begin{proof} Since both $M_{\cG}(\mu)$ and $L^{+, Q(u)} \cG$ are flat over $\Spec \cO_E$, this follows from the fact that $M_{\cG}(\mu)_E$ is stable under $(L^{+, Q(u)} \cG)_E$. 
\end{proof}  

The idea going back to \cite{PR1, PR2} is that the special fiber should be determined by the ``sum" of the components of $\{ \mu \}$ over the embeddings of $\psi:K \ra \overline{F}$.  To define this sum, choose a Borel subgroup $B$ of $G_{\overline{F}}$ which contains $T_{\overline{F}}$.   Let $\mu$ be the unique representative of $\{ \mu \}$ valued in $T_{\overline{F}}$ such that for every embedding $\psi:K \ra \overline{F}$ the component $\mu_{\psi}$ is $B$-dominant.  Define
$$
\la_{\mu} := \sum_{\psi:K \ra \overline{F}} \mu_{\psi}.
$$
The conjugacy class of $\la_{\mu}$ is independent of choice of $B$.  

We are now ready to state the theorem identifying the special fiber of $M_{\cG}(\mu)$.  Recall the $\la_{\mu}$-admissible set $\Adm^{\fa^{\flat}}(\la_{\mu})$ introduced by Kottwitz and Rapoport (see Definition \ref{admissible}). 

\begin{thm} \label{thmadm} Suppose that $p \nmid |\pi_1(G_{\der})|$.  As a closed subscheme of $\Fl_{\cP_{\fa^{\flat}}}$, we have
$$
\bigcup_{w \in \Adm^{\fa^{\flat}} (\la_{\mu})} S_w^{\fa^{\flat}} = \overline{M}_{\cG}(\mu).
$$
\end{thm} 

\begin{rmk} The set $\Adm^{\fa^{\flat}} (\la_{\mu})$ only depends on the geometric conjugacy class of $\la_{\mu}$ so the right side in Theorem \ref{thmadm} like the left side is determined only by $\{\mu\}$. 
\end{rmk}

We begin by constructing an ample line bundle on $M_{\cG}(\mu)$. Consider the case of $\cG = \GL_n$.  Let $M_0 = (\cO[u])^n$ be as in Example \ref{GLN}.  For any $Q(u)$-lattice $M$ of $M_0[1/Q(u)]$ with coefficients in $R$, define
$$
\det_R(M) := \det_R(Q(u)^{-N}(M_0 \otimes R)/M) \otimes_R \det_R((M_0 \otimes R)/Q(u)^N (M_0 \otimes R))
$$
for $N \gg 0$.  See \cite[10.1.15-17]{LevinThesis} for more detail. 

\begin{defn} \label{GLNdet} Define a line bundle $\cL_{\det}^{Q(u)}$ on $\Gr_{\GL_n}^{Q(u)}$ by 
 $$
M \mapsto \det_R(M)
$$
for any $M \in \Gr_{\GL_n}^{Q(u)}(R)$ (using the description of $\Gr_{\GL_n}^{Q(u)}$ from Example \ref{GLN}).  
\end{defn}

\begin{rmk} The line bundle $\cL_{\det}^{Q(u)}$ is defined in the same way as the determinant line bundle on the ordinary affine Grassmanian $\Gr_{\GL_n}$.  It is not hard to see that $\cL_{\det}^{Q(u)}$ is an ample line bundle from the description of  $\Gr_{\GL_n}^{Q(u)}$ as an ind-scheme from Example \ref{GLN}.
\end{rmk} 

Let us recall briefly a few facts about line bundles on ordinary affine Grassmanians.   Let $H$ be a connected reductive groups over an algebraically closed field $\kappa$.  The affine Grassmanian $\Gr_{\GL_n}$ is the moduli space of ``lattices" in $\kappa(\!(u)\!)^n$  and is equipped with a canonical ample line bundle the determinant line bundle $\cL_{\det}$ (see \cite[pg.42]{Faltings}).  

Let $\Lie (H)$ denote the Lie algebra of $H$.  Define $\cL_{H, \det}$ on $\Gr_H$ to be the pullback of $\cL_{\det}$ under the natural map
$$
\Ad:\Gr_H \ra \Gr_{\GL(\Lie (H))}.
$$ 
More generally, if $\cP$ is a parahoric group scheme over $\kappa[\![u]\!]$, then $\Lie (\cP)$ is a finite free $\kappa[\![u]\!]$-module.  We define $\cL_{\cP, \det}$ on $\Fl_{\cP}$ to be the pullback of $\cL_{\det}$ under
$$
\Ad:\Fl_{\cP} \ra \Gr_{\GL(\Lie (\cP))}.
$$ 

\begin{prop} \label{Pample} Let $\cP$ be a parahoric group scheme over $\kappa[\![u]\!]$ such that $G' := \cP_{\kappa(\!(u)\!)}$ splits over a tamely ramified extension.  Assume that $\mathrm{char}(\kappa) \nmid |\pi_1(G'_{\der})|$.  Then $\cL_{\cP, \det}$ is ample on $\Fl_{\cP}$.  
\end{prop}
\begin{proof}  The determinant line bundle is translation invariant by $L G'(\kappa)$ (\cite[pg. 43]{Faltings} or \cite[Lemma 10.3.2]{LevinThesis}) so it suffices to show that $\cL_{\cP, \det}$ restricted to the neutral connected component $\Fl^0_{\cP}$ is ample.   Let $G'_{\der}$ be the derived subgroup of $G'$ and let $\cP_{\der}$ be the corresponding parahoric for $G'_{\der}$.  By \cite[Proposition 6.6]{PRTwisted},
$$
\left(\Fl_{\cP_{\der}}^0 \right)_{\red} \cong  \left(\Fl_{\cP}^0 \right)_{\red}
$$
so we can reduce to the case where $G'$ is semi-simple.  

Let $\widetilde{G}' \ra G'$ denote the simply-connected cover of $G'$.   Let $\widetilde{\cP}$ be the corresponding parahoric for $\widetilde{G}'$. By  \cite[\S 6.a]{PRTwisted} especially equations (6.7) and (6.11), we have 
$$
 \Fl_{\widetilde{\cP}} \cong   \Fl_{\cP}^0
$$
using also that $ \Fl_{\widetilde{\cP}}$ is connected\cite[Theorem 5.1]{PRTwisted}. Thus, we are reduced to the case where $G'$ is simply-connected.  

Then, $G' \cong \prod \Res_{F_i/\kappa(\!(u)\!)} G_i'$ where $F_i = \kappa(\!(v_i)\!)$ are tamely ramified extensions of $\kappa(\!(u)\!)$ and $G_i'$ are absolutely simple and simply connected.  One can then reduce to the case of $G' = G_i'$.  Choose a Iwahori subgroup $\cI$ contained in $\cP$.  If $\mathbf{S}$ is the set of vertices of the affine Dynkin diagram of $G'$, then $\cP = \cP_I$ for some non-empty subset of $I \subset \mathbf{S}$.  In this case, we have $\Pic(\Fl_{\cP_I}) \cong \Z \cL(\eps_i)^I$ where the coefficient on $\cL(\eps_i)$ is the degree of the line bundle restricted to the projective line $\bP_i$ which is the image of $\cP_{\mathbf{S} - i}/\cI$ (see \cite[Proposition 10.1]{PRTwisted} or discussion after Theorem 2.3 in \cite{ZhuCoh}).

We just have to check that $\cL_{G', \det}$ is ample on $\bP_i$.  When $I = \mathbf{S}$, this is a consequence of \cite[Lemma 4.2]{ZhuCoh}.  For a general parahoric, we have
$$
\bP_i \subset \cP_{I -i}/\cP_I \subset \Fl_{\cP_I}.  
$$
Let $\overline{\cP}_{I-i}$ denote the special fiber of the parahoric $\cP_{I- i}$.  As in the proof of \cite[Lemma 4.2]{ZhuCoh}, $\cP_I$ defines a maximal proper parabolic $P_I$ of the reductive quotient $M_{I-i} := \overline{\cP}^{\red}_{I-i}$.  We can identify $\cP_{I -i}/\cP_I$ with $M_{I-i}/P_I$.  The determinant line bundle on $\Fl_{\cP_I}$ restricted to $M_{I-i}/P_I$ is isomorphic to the determinant line bundle on the Grassmanian of $\dim_{\kappa} \Lie P_I$ subspaces $\Lie M_{I-i}$ which is ample.   
\end{proof}

Now we return to our Bruhat-Tits group scheme $\cG$.  Let $\cV = \Lie \cG$ a finite projective $\cO[u]$-module.  By \cite{Seshadri}, $\cV$ is a free $\cO[u]$-module.  The adjoint representations induces a natural map 
$$
\Ad:\Fl_{\cG}^{Q(u)} \ra \Gr_{\GL(\cV)}^{Q(u)}.
$$
Hence we have a natural line bundle
$$
\cL^{Q(u)}_{\cG} := \Ad^*(\cL^{Q(u)}_{\det})
$$
given by pulling back the determinant line bundle.  

\begin{prop} \label{Ldecomp} The line bundle $\cL^{Q(u)}_{\cG}$ is ample on $\Fl_{\cG}^{Q(u)}$.  Furthermore, under the isomorphism 
$$
\left(\Fl_{\cG}^{Q(u)}\right)_{\overline{F}} \cong \prod_{\psi:K \ra \overline{F}} (\Gr_{(G_{\psi, \overline{F}})}
$$
from Proposition \ref{genfiber} we have $(\cL^{Q(u)}_{\cG})_{\overline{F}} \cong \boxtimes_{\psi:K \ra \overline{F}} (\cL_{G, \det})_{\psi, \overline{F}}$.
\end{prop} 
\begin{proof}
Since $\Fl_{\cG}^{Q(u)}$ is ind-projective (Theorem \ref{AGproj}), we can check ampleness on fibers.  On the special fiber, 
$$
(\cL^{Q(u)}_{\cG})_k = \cL_{\cP_{\fa^{\flat}}, \det}
$$
under the isomorphism from Proposition \ref{specfiber} which is ample by Proposition \ref{Pample}. 

The second statement of the Proposition will imply that $\cL^{Q(u)}_{\cG}$ is ample on the generic fiber as well.  The product decomposition for $(\cL^{Q(u)}_{\cG})_{\overline{F}}$ reduces immediately to the case where $\cG = \GL_n$. The ind-scheme $(\Gr^{Q(u)}_{\GL_n})_{\overline{F}}$ is a moduli space of $Q(u)$-lattices in $M_0 = \overline{F}[u]^n$  (Example \ref{GLN}).  The product decomposition 
$$
\left(\Gr_{\GL_n}^{Q(u)}\right)_{\overline{F}} \cong \prod_{\psi:K \ra \overline{F}} (\Gr_{\GL_n})_{\psi, \overline{F}}
$$
is induced by the decomposition $$(Q(u)^{-N} M_0/Q(u) M_0) \cong \bigoplus_{\psi}\left( (u- \psi(\varpi))^{-N} M_0 / (u - \psi(\varpi))^N M_0 \right)$$ for all positive integers $N$.  One computes directly that $\cL_{\det}^{Q(u)} \cong \boxtimes_{\psi:K \ra \overline{F}} \cL_{\det}$ as in \cite[Proposition 10.1.19]{LevinThesis}.
\end{proof} 

If $X$ is a projective scheme over a field $\kappa$ and $\cL$ is a line bundle, the $h^0(X, \cL)$ is defined to be the $\kappa$-dimension of $H^0(X, \cL)$. 

\begin{prop} \label{productformula} Let $H$ be split connected reductive group over a field $\kappa$ with Borel subgroup $B_H$. Suppose that $\mathrm{char}(\kappa) \nmid |\pi_1(H_{\der})|$.  Then
$$
h^0(S_H(\mu + \mu'), \cL_{H, \det}^{\otimes n}) = (h^0(S_H(\mu), \cL_{H, \det}^{\otimes n})) (h^0(S_H(\mu'), \cL_{H, \det}^{\otimes n}))
$$
for any $B_H$-dominant cocharacters $\mu$ and $\mu'$ of $H$. 
\end{prop}
\begin{proof}  Under the hypothesis on $\pi_1(H_{\der})$, the Schubert varieties and $\cL_{H, \det}$ are all defined over $\Spec \Z$ by \cite{Faltings} (see also the discussion in \cite[\S 8.e.3, \S 8.e.4]{PRTwisted}) and $H^1(S_H(\la), \cL_{H, \det}) = 0$ for all cocharacters $\la$ since $S_H(\la)$ is Frobenius split in finite characteristic.  Thus, we can reduce to the case of characteristic 0. In this setting, it is originally due to \cite[Theorem 1]{FL}.  A geometric proof is given in \cite[Theorem 1.2.2]{ZhuAffine}.   
\end{proof}

\begin{rmk} The product formula in Proposition \ref{productformula} plays a role like that of the coherence conjecture of Pappas-Rapoport when the group is defined by Weil-restriction.  For an unramified group and certain maximal parahorics, one only needs the product formula to prove Theorem \ref{thmadm} as in \cite{LevinThesis}.    
\end{rmk}

\begin{proof}[Proof of Theorem \ref{thmadm}:]  Set $\cA^{\cP_{\fa^{\flat}}}(\la_{\mu}) := \cup_{w \in \Adm^{\fa^{\flat}} (\la_{\mu})} S_w^{\fa^{\flat}}$.  We first show that
\begin{equation} \label{w5}
\cA^{\cP_{\fa^{\flat}}}(\la_{\mu}) \subset \overline{M}_{\cG}(\mu).
\end{equation}
Since $ \overline{M}_{\cG}(\mu)$ is $L^+ \cP_{\fa^{\flat}}$-stable, it suffices to show that $S_w^{\fa^{\flat}} \subset \overline{M}_{\cG}(\mu)$ for $w$ an extremal element in $\Adm^{\fa^{\flat}} (\la_{\mu})$ under the Bruhat order.  Representatives for the extremal elements in $\Adm^{\fa^{\flat}} (\la_{\mu})$ are given by the translations $t_{\la'}$ where $\la'$ is in the $W$-orbit of $\la_{\mu}$ (see \cite[Corollary 2.9]{RicharzAG}).  The inclusion (\ref{w5}) follows from Corollary \ref{inclusion}.  

Let $\cL := \cL_{\cG}^{Q(u)}$ be the ample line bundle on $M_{\cG}(\mu)$ from Proposition \ref{Ldecomp}.   To show that 
$$
\cA^{\cP_{\fa^{\flat}}}(\la_{\mu}) = \overline{M}_{\cG}(\mu)
$$
it suffices to show that 
$$
h^0(\cA^{\cP_{\fa^{\flat}}}(\la_{\mu}), \cL_k^{\otimes n}) = h^0(\overline{M}_{\cG}(\mu), \cL_k^{\otimes n})
$$
for all $n \gg 0$.  By flatness, for $n \gg 0$, 
$$
h^0(\overline{M}_{\cG}(\mu), \cL_k^{\otimes n}) = h^0(M_{\cG}(\mu)_{\overline{F}}, \cL_{\overline{F}}^{\otimes n}).
$$
Let $H$ be the split form of $G$. Furthermore, $M_{\cG}(\mu)_{\overline{F}} = \prod_{\psi:K \ra \overline{F}} S_{H_{\overline{F}}} (\mu_{\psi})$ and so by Proposition \ref{Ldecomp}
$$
h^0(M_{\cG}(\mu)_{\overline{F}}, \cL_{\overline{F}}^{\otimes n}) = \prod_{\psi:K \ra \overline{F}} h^0(S_{H_{\overline{F}}} (\mu_{\psi}), \cL_{\det, G}^{\otimes n}).
$$
The product formula (Proposition \ref{productformula}) implies that 
$$
 \prod_{\psi:K \ra \overline{F}} h^0(S_{H_{\overline{F}}} (\mu_{\psi}), \cL_{\det, G}^{\otimes n}) = h^0(S_G(\la_{\mu})_{\overline{F}}, \cL_{\det, G}^{\otimes n}).
$$
Finally we appeal to the coherence conjecture of Pappas-Rapoport in the form given in \cite[(9.19)]{PZ} (the conjecture is proven in \cite{ZhuCoh}) which says that 
$$
h^0(\cA^{\cP_{\fa^{\flat}}}(\la_{\mu}), \cL_k^{\otimes n})  = h^0(S_G(\la_{\mu})_{\overline{F}}, \cL_{\det, G}^{\otimes n})
$$
for $n \gg 0$.  
\end{proof}

\begin{rmk} Theorem \ref{thmadm} implies that the irreducible components of $\overline{M}_{\cG}(\mu)$ are in bijection with extremal elements in $\Adm^{\fa^{\flat}}(\la_{\mu})$.   Thus, $\overline{M}_{\cG}(\mu)$ is equidimensional and the number of components is equal to the size of $W_{0, \fa^{\flat}} \backslash W_0 / W_{0, \la_{\mu}}$ by \cite[Corollary 2.9]{RicharzAG} where $W_{0, \la_{\mu}}$ is the stabilizer of image of $\la_{\mu}$ in $X_*(T)_I$.  
\end{rmk}

\section{Nearby cycles}

\subsection{Setup and constructions}

In this section, we study the sheaves of nearby cycles of the local models we have constructed.  This extends work of \cite{Gaitsgory, HNnearby, PZ}.  We show in particular Theorem \ref{Kottwitz} (Kottwitz conjecture) and the unipotence of the monodromy action (Theorem \ref{monodromy}).  In \S 5.4, we introduce splitting models when $G$ is unramified and $\fa$ is \emph{very special}. In this situation, we are able to give a more explicit description of the sheaf of nearby cycles generalizing results of \cite{PR1, PR2}.  

To simplify notation, we will assume that $K/F$ is totally ramified (i.e., $K_0 = F$).  Everything we say holds in the situation where $K/F$ is not totally ramified using that $M_{\cG}(\mu)_{\cO_{K_0}}$ is a product of local models for the group $\Res_{K/K_0} G$.   

Our notation and setup will be as in \S 10.1.1 and 10.1.2 of \cite{PZ} unless otherwise indicated.  We refer the reader there for background on perverse sheaves and nearby cycles. In particular, $X = \bA^1_{\cO} = \Spec \cO[u]$ and 
$$
\cP = \cG \times_X \Spec (\cO[\![t]\!]) \text{ with } u \mapsto t
$$
is a group scheme over $\cO[\![t]\!]$. We define $\Fl_{\cP}$ to be the ``twisted" affine flag variety over $\Spec \cO$ parametrizing $\cP$-bundles trivialized away from $t = 0$. Note that $\cP_{F[\![t]\!]}$ is a parahoric group scheme for the group $\cG_{F(\!(t)\!)}$ and $\cP_{k[\![t]\!]}$ is the parahoric group scheme $\cP_{\fa^{\flat}}$ from Definition \ref{xuuu}. For any $\cO$-algebra $\kappa$, we will use $\cP_{\kappa}$ to denote the base change $\cP_{\kappa[\![t]\!]}$.   For any finite field $k' \supset k$, we have
$$
\cP_{k'} = (\cP_{\fa^{\flat}})_{k'}.
$$

\begin{rmk} The field $K$ doesn't play a a role in the definition of $\Fl_{\cP}$ which is essentially the same as in \cite[\S 10.1.2]{PZ}.  In particular, \cite[Lemma 10.4]{PZ} concerning the $\IC_w$ holds in our situation as well. 
\end{rmk}

We now define the main object of study in this section.

\begin{defn} Define $\cF_{\mu}$ to be the intersection cohomology sheaf on the generic fiber $M_{\cG}(\mu)_E$. The nearby cycles of $\cF_{\mu}$ is denoted
$$
R \Psi_{\mu} := R \Psi^{M_{\cG}(\mu)}(\cF_{\mu}).
$$
\end{defn} 

Since the nearby cycles functor preserves perversity, $R \Psi_{\mu}$ is a perverse sheaf on $(M_{\cG}(\mu))_{\overline{k}}$ with an action of $\Gal(\overline{F}/E)$.  The main theorem (Theorem \ref{commconst}) in this section will be the commutativity constraint on $R \Psi_{\mu}$. 
\begin{prop} \label{sheafequiv} The perverse sheaf $R \Psi_{\mu}$ on $\overline{M}_{\cG}(\mu)_{\overline{k}}$ admits a natural $(L^+ \cP)_{\overline{k}} = (L^+ \cP_{\fa^{\flat}})_{\overline{k}}$-equivariant structure as perverse sheaves on $(\Fl_{\cP})_{\overline{k}}$ which is compatible with the Galois action of $\Gal(\overline{F}/E)$ $($see \cite[Lemma 10.2]{PZ} for a more precise statement$)$.
\end{prop}
\begin{proof} The proof is the same as in \cite[Lemma 10.2]{PZ} except one replaces the jet group $L^+_n \cG$ by the quotient of $L^{+, Q(u)} \cG$ given by 
$$
L^{+, Q(u)}_n \cG(R) = \cG(R[u]/Q(u)^n)
$$
which is represented by a smooth affine group over $\Spec \cO$.
\end{proof}

As in \cite[Definition 10.3]{PZ}, we will use $\Perv_{L^+ \cP_{k_E}}(\Fl_{\cP_{k_E}} \times_{k_E} E, \overline{\Q}_{\ell})$ to denote the category of $(L^+ \cP)_{\overline{k}}$-equivariant sheaves on $(\Fl_{\cP})_{\overline{k}}$ with a compatible action of $\Gal(\overline{F}/E)$.  By Proposition \ref{sheafequiv}, $R \Psi_{\mu}$ is an object of $\Perv_{L^+ \cP_{k_E}}(\Fl_{\cP_{k_E}} \times_{k_E} E, \overline{\Q}_{\ell})$.  

The following are generalizations of the constructions from \cite[\S 10.2]{PZ} which, in turn, are mixed characteristic versions of constructions from \cite{Gaitsgory, BD}.  

\begin{defn} \label{BDdef} Let $D$ be the divisor on $X$ defined by $Q(u) = 0$ and $D_0$ the divisor defined by $u = 0$. For any $\cO$-scheme $S$, we define
\begin{equation*} \label{BDGrass}
\Gr_{\cG}^{\BD, Q(u)}(S) = \{ \text{iso-classes of pairs} (\cE, \beta) \}
\end{equation*}
where $\cE$ is a $\cG$-bundles on $X \times S$ and $\beta$ is a trivialization of $\cE|_{X_S \backslash (D \cup D_0) }$,
\begin{equation*} \label{ConvGrass}
\Gr_{\cG}^{\Conv, Q(u)}(S) = \{ \text{iso-classes of } (\cE, \cE', \beta, \beta') \}
\end{equation*}
where $\cE, \cE'$ are two $\cG$-bundles on $X \times S$, $\beta$ is a trivialization of $\cE|_{X_S \backslash D}$, $\beta'$ is an isomorphism of $\cE'|_{X_S \backslash D_0} \cong \cE|_{X_S \backslash D_0}$, and
\begin{equation*} \label{ConvGrass2}
\Gr_{\cG}^{\Conv', Q(u)}(S) = \{ \text{iso-classes of } (\cE, \cE', \beta, \beta') \}
\end{equation*}
where $\cE, \cE'$ are two $\cG$-bundles on $X \times S$, $\beta$ is a trivialization of $\cE|_{X_S \backslash D_0}$, $\beta'$ is an isomorphism of $\cE'|_{X_S \backslash D} \cong \cE|_{X_S \backslash D}$.
\end{defn}

 We will now describe the fibers of each of the above constructions in turn.  

\begin{prop} \label{BDfibers}  Let $\Gr_{\cG, F}^{\BD, Q(u)}$ and $\Gr_{\cG, k}^{\BD, Q(u)}$ denote the generic and special fibers respectively of $\Gr_{\cG}^{\BD, Q(u)}$. We have natural isomorphisms
$$
\Gr_{\cG, F}^{\BD, Q(u)} \cong \Gr_{\cP_F} \times \Fl_{\cG, F}^{Q(u)}  \text{ and } \Gr_{\cG, k}^{\BD, Q(u)} \cong \Gr_{\cP_k} \cong \Fl_{\cP_{\fa^{\flat}}}.
$$
\end{prop}
\begin{proof} The isomorphism on the special fiber is clear since $D \cup D_0 = D_0$ on $\bA^1_{k}$. The isomorphism on the generic fiber follows from the gluing argument used in the proof of Proposition \ref{genfiber} (see also \cite[Proposition 5]{Gaitsgory}).
\end{proof} 

There is a natural map 
$$
m:\Gr_{\cG}^{\Conv, Q(u)} \ra \Gr_{\cG}^{\BD, Q(u)}
$$
sending $(\cE, \cE', \beta, \beta')$ to $(\cE', \beta \beta')$.  Similarly, there is a natural map 
$$
m':\Gr_{\cG}^{\Conv', Q(u)} \ra \Gr_{\cG}^{\BD, Q(u)}
$$
defined by the same formula.  

Recall that for $\cP_k$ a parahoric group scheme over $k[\![u]\!]$, the twisted product $\Gr_{\cP_k} \widetilde{\times} \Gr_{\cP_k}$ is the quotient $L \cP_k \times^{L^+ \cP_k} \Gr_{\cP_k}$ of the ind-scheme $L \cP_k \times \Gr_{\cP_k}$.  It sits in the convolution diagram
\begin{equation} \label{convdiag}
\Gr_{\cP_k} \times \Gr_{\cP_k} \xleftarrow{p} L \cP_k \times \Gr_{\cP_k} \xrightarrow{q} \Gr_{\cP_k} \widetilde{\times} \Gr_{\cP_k} \rightarrow \Gr_{\cP_k}
\end{equation}
where both $p$ and $q$ are $L^+ \cP_k$-torsors (hence formally smooth).  The twisted product $\Gr_{\cP_k} \widetilde{\times} \Gr_{\cP_k}$ represents the functor defined in \cite[Lemma 1]{Gaitsgory}. 

\begin{prop} \label{Convfibers} Both $m$ and $m'$ become isomorphisms over $F$.  We have natural isomorphisms
$$
\Gr_{\cG, k}^{\Conv, Q(u)} \cong \Gr_{\cP_k} \widetilde{\times} \Gr_{\cP_k} \text{ and }  \Gr_{\cG, k}^{\Conv', Q(u)} \cong \Gr_{\cP_k} \widetilde{\times} \Gr_{\cP_k} 
$$
such that $m \otimes_{\cO} k$ and $m' \otimes_{\cO} k$ induce the convolution diagram $\Gr_{\cP_k} \widetilde{\times} \Gr_{\cP_k}  \ra \Gr_{\cP_k}$ $($see \cite[(10.2), (10.3)]{PZ}$)$.
\end{prop}
\begin{proof}
We prove that $m_{F}$ is an isomorphism.  An analogous argument works for $m'_{F}$. Over $F$, $D$ and $D_0$ are disjoint divisors. We construct the inverse to $m_{F}$ as follows. For any $(\cE', \alpha') \in  \left(\Gr_{\cG}^{\BD, Q(u)} \right)_F(R)$, define $\cE$ to be the bundle on $X_R$ defined by gluing the trivial bundle on $X_R \backslash D$ to $\cE'_{X_R \backslash D_0}$ with gluing data given by $\alpha'$.  The bundle $\cE$ comes equipped with a trivialization $\beta$ over $X_R \backslash D$ and an isomorphism $\beta'^{-1}:\cE|_{X_R \backslash D} \cong \cE'|_{X_R \backslash D}$.  
  
Over $k$, we have $D_0 = D$ and so $\Gr_{\cG, k}^{\Conv, Q(u)}$ represents the same functor as $\Gr_{\cP_k} \widetilde{\times} \Gr_{\cP_k}$ (see \cite[Lemma 1]{Gaitsgory}).  
\end{proof}    
  
\begin{prop} \label{indproper3} All three functors $\Gr_{\cG}^{\BD, Q(u)}, \Gr_{\cG}^{\Conv, Q(u)}$, and $\Gr_{\cG}^{\Conv', Q(u)}$ are represented by ind-schemes over $\Spec \cO$ which are ind-proper.
\end{prop}
\begin{proof}   
One can prove that $\Gr_{\cG}^{\BD, Q(u)}$ is ind-proper using the same argument which showed that $\Gr_{\cG}^{Q(u)}$ was ind-proper (see Theorem \ref{AGproj}). 

Consider the projection map $p:\Gr_{\cG}^{\Conv, Q(u)} \ra \Gr_{\cG}^{Q(u)}$ defined by the forgetful map 
$$
(\cE, \cE', \beta, \beta') \mapsto (\cE, \beta).
$$ 
The map $p$ is a fibration with fibers isomorphic to $\Gr_P$.   That is, fppf-locally $\Gr_{\cG}^{\Conv, Q(u)}$ is a product of $\Gr_P$ and $\Gr_{\cG}^{Q(u)}$ and hence ind-proper. The same argument works for $\Gr_{\cG}^{\Conv', Q(u)}$ using the natural projection $p'$ onto $\Gr_P$ with fibers isomorphic to $\Gr_{\cG}^{Q(u)}$.  
\end{proof}

\subsection{A commutativity constraint} 

Let $w \in \widetilde{W} = \widetilde{W}^{\flat}$ as in \S 3.3.   There is an affine Schubert variety $S_w \in (\Gr_{\cP})_{\cO'}$ where $F'$ is an unramified extension of $F$ with residue field $k'$. Assume $k' \supset k_E$ and let $E'$ denote the unramified extension of $E$ with residue field $k'$. We denote the intersection cohomology sheaf on $S_w$ by $\IC_w$ which satisfies \cite[Lemma 10.2]{PZ}.    

The main theorem of this section is a version of Theorem 10.5 from \cite{PZ}:
\begin{thm} \label{commconst} There is a canonical isomorphism 
$$
c_{\cF}:\IC_{w, \overline{k}} \star R \Psi_{\mu} \xrightarrow{\sim} R \Psi_{\mu} \star \IC_{w, \overline{k}}
$$
of perverse sheaves on $\Gr_{\cP_{\overline{k}}}$. In addition, this isomorphism respects the $\Gal(\overline{F}/E')$ action on both sides. 
\end{thm}

As in \cite[\S 10.2]{PZ}, Theorem \ref{commconst} is a consequence of the following identities involving the constructions from the previous subsection:
\begin{prop} \label{commid} Assume $\IC_{w, \overline{k}}$ is defined over $k'$.  We have canonical isomorphisms in the category $\Perv_{L^+ \cP_{k'}}(\Gr_{\cP_{k'}} \times_{k'} E', \overline{\Q}_{\ell})$:
\begin{enumerate}
\item[$(1)$] $R \Psi^{\Gr^{\BD, Q(u)}_{\cG, \cO_{E'}}} \left( \IC_{w, E'} \boxtimes \cF_{\mu} \right) \xrightarrow{\sim} R \Psi_{\mu} \star IC_{w, k'}$
\item[$(2)$] $R \Psi^{\Gr^{\BD, Q(u)}_{\cG, \cO_{E'}}} \left( \IC_{w, E'} \boxtimes \cF_{\mu} \right) \xrightarrow{\sim} IC_{w, k'} \star R \Psi_{\mu}$.
\end{enumerate}
\end{prop} 
\begin{proof} The proof is essentially the same as the proof of \cite[Proposition 10.7]{PZ} replacing their constructions by the ones from the previous subsection.  We highlight the main points. For (1), we regard $IC_{w, E'} \boxtimes \cF_{\mu}$ as a sheaf on $\Gr^{\Conv, Q(u)}_{\cG, E'}$ via the isomorphism $m \otimes_{\cO} E'$ and the isomorphism from Proposition \ref{BDfibers}.   By Proposition \ref{Convfibers}, $m \otimes_{\cO} k'$  is the convolution diagram.  Since $m$ is proper and vanishing cycles commute with proper pushforward, it suffices to give an isomorphism
\begin{equation} \label{fff}
R \Psi^{\Gr^{\Conv, Q(u)}_{\cG, \cO_{E'}}} \left( \IC_{w, E'} \boxtimes \cF_{\mu} \right) \xrightarrow{\sim} R \Psi_{\mu} \widetilde{\times} IC_{w, k'}
\end{equation}
where $R \Psi_{\mu} \widetilde{\times} IC_{w, k'}$ is the twisted product of $L^+ \cP_{k'}$-equivariant sheaves.

As in \cite[Proposition 10.7]{PZ}, $\Gr^{\Conv, Q(u)}_{\cG, \cO_{E'}}$ is the twisted product $\widetilde{\Gr}_{\cG}^{Q(u)} \times^{L^+_n \cP} \Gr_{\cP}$, where $\widetilde{\Gr}_{\cG}^{Q(u)}$ is the smooth cover of $\Gr_{\cG}^{Q(u)}$ obtained by adding a trivialization of $\cE$ over the $n$th infinitessimal neighborhood of $D_0$.  Since the support of $\IC_{w, E'} \boxtimes \cF_{\mu}$ is finite type, for some $n$ sufficiently large  $\IC_{w, E'} \boxtimes \cF_{\mu}$ is supported on $\widetilde{\Gr}_{\cG}^{Q(u)} \times^{L^+_n \cP} S_w \subset \Gr^{\Conv, Q(u)}_{\cG, \cO_{E'}}$.  The isomorphism \ref{fff} follows since the formation of nearby cycles commutes with smooth base change combined with \cite[Lemma 10.4]{PZ}. 

The argument for (2) is similar using $\Gr^{\Conv, Q(u)}_{\cG, \cO_{E'}}$.  In this case, we work with $\widetilde{\Gr}_{\cP}$ defined by adding trivialization along the $n$th infinitessimal neighborhood of $D$.  Then, $L^+ \cP$ is replaced by the completion of $\cG$ at $Q(u)$ ($L^{+, Q(u)} \cG$ from the discussion after Proposition \ref{indscheme}). Otherwise, all the details are the same.  
\end{proof}   

\subsection{Monodromy of the nearby cycles} 

In this section, we derive some consequences of Theorem \ref{commconst}. For the most part, the proofs are the same as in \cite{PZ} where we refer the reader for details. 

Let $\widetilde{F}$ be a Galois closure of $\widetilde{K}/F$ in $\overline{F}$.  Since $\Res_{K/F} G$ splits over $\widetilde{F}$, for any $\{ \mu \}$ the reflex field $E$ is a subfield of $\widetilde{F}$.  Let 
$$
I_E = \ker(\Gal(\overline{F}/E) \ra \Gal(\overline{k}/k_E))
$$ 
which is the ``monodromy'' group.  We first study the action of the inertia group $I_{\widetilde{F}}$ on $R \Psi_{\mu}$.  Set $\widetilde{M}_{\cG}(\mu) := M_{\cG}(\mu)_{\widetilde{F}}$. Define
$$
\widetilde{R \Psi}_{\mu} = R \Psi^{\widetilde{M}_{\cG}(\mu)} \IC_{\mu}
$$
where $\IC_{\mu}$ is the intersection cohomology sheaf on $S(\mu)$. Then, $\widetilde{R \Psi}_{\mu}$ is the same as $R \Psi_{\mu}$ with the Galois action restricted to $\Gal(\overline{F}/\widetilde{F})$ since the pullback of $\cF_{\mu}$ to $\widetilde{M}_{\cG}(\mu)_{\widetilde{F}}$ is $\IC_{\mu}$.  

Recall the notion of very special facet introduced in \cite{ZhuGS}. A facet $\fa \in \cB(\Res_{K/F} G, F)$ is \emph{very special} if it is special and remains special over the completion of the maximal unramified extension of $F$ (equivalently over the maximal unramified subextension of $\widetilde{F}/F$). 

\begin{thm}\label{monodromy}  The action of $I_{\widetilde{F}}$ on the nearby cycles $R \Psi_{\mu}$ is unipotent.  Assume that $\fa \in B(\Res_{K/F} G, F)$ is a very special vertex.  Then the action of $I_{\widetilde{F}}$ on the nearby cycles $R \Psi_{\mu}$ is trivial.  
\end{thm}
\begin{proof} For the first statement, the proof of \cite[Theorem 10.9]{PZ} only uses the commutativity constraint and the theory of central sheaves on affine flag varieties.  Given Theorem \ref{commconst} the proof goes through as written there.  The second statement is the same as \cite[Proposition 10.12]{PZ} and the proof is the same. 
\end{proof} 

\begin{rmk} When $G= \GL_n$ and $\mu$ is a Shimura (minuscule) cocharacter, Theorem \ref{monodromy} was proven in the case of a special vertex in \cite[\S 7]{PR1}.  In addition, they give an explicit description of the action of $I_E$ on $R \Psi_{\mu}$.  We will study the monodromy action in more detail in the next section. 
\end{rmk}

Finally, we consider the semi-simple trace of Frobenius on $R \Psi_{\mu}$.  We refer to \cite[\S 10.4.1]{PZ} and \cite[\S 3.1]{HNnearby} for details.  Recall that $G^{\flat} = \cG_{k(\!(u)\!))}$ and $\cP_{k} = \cG_{k[\![u]\!]}$, a parahoric group scheme.  For any $\F_q \supset k$, we let $\cH_q(G^{\flat}, P^{\flat})$ be the Hecke algebra of bi-$\cP_{k}(\F_q[\![u]\!])$-invariant, compactly supported locally constant $\overline{\Q}_{\ell}$-valued functions on $G^{\flat}(\F_q(\!(u)\!))$ which is algebra under convolution.  Then the semi-simple trace defines a map
$$
 \tau^{\mathrm{ss}}:\Perv_{L^+ \cP_{k_E}}(\Fl_{\cP_{k_E}} \times_{k_E} E, \overline{\Q}_{\ell}) \ra \cH_q(G^{\flat}, \cP_k)
$$
for any $\F_q \supset k_E$.  

\begin{thm} \label{Kottwitz} If $\F_q \supset k_E$, then $\tau^{\mathrm{ss}}_{R\Psi_{\mu}}$ is in the center of $\cH_q(G^{\flat}, P^{\flat})$.
\end{thm}
\begin{proof} The theorem follows from Theorem \ref{commconst} as in \cite[\S 8]{HNnearby}.  
\end{proof}

When $G$ is split and $\mu$ is a Shimura cocharacter, this is known as the Kottwitz conjecture.  It was proven in \cite{HNnearby} for $\GL_n$ and $\GSp_{2g}$.

\subsection{Unramified groups and splitting models}

In this section, we assume $G$ is an unramified group over $K$ (i.e., quasi-split and split over an unramified extension) as in \S 2. Let $K'$ be a unramified extension which splits $G$.  We also assume that $\fa$ is a hyperspecial vertex of $\cB(G, K)$.  Note that in this case $\Res_{K/F} G$ is quasi-split and $\fa$ considered as vertex of $\cB(\Res_{K/F} G, F)$ is \emph{very special} in the sense of \cite[\S 10.3.2]{PZ}. 

Let $\widetilde{F} \subset \overline{F}$ be a Galois closure of $K'$ in $\overline{F}$ over which $\Res_{K/F} G$ becomes split. Fix an ordering of the $F$-embeddings $\psi_i:K \ra \widetilde{F}$ of which there are $d = [K:F]$.    By Proposition \ref{monodromy}, the inertia group $I_{\widetilde{F}}$ acts trivially on $\widetilde{R \Psi}_{\mu}$.  

Let $H$ be the split form of $G$ defined over $\cO_F$, i.e., $H_{K'} \cong G_{K'}$.  Since $G$ is unramified and $\fa$ is very special, the special fiber $\Gr_{\cG}^{Q(u)} \otimes \overline{k} \cong \Gr_{H_{\overline{k}}}$.  Also, for any embedding $\psi_i$, we have an isomorphism $H_{\widetilde{F}} \cong G_{\widetilde{F}, \psi_i}$. If $\mu$ is geometric cocharacter of $\Res_{K/F} G$, let $\mu_i$ be the $\psi_i$-component of $\mu$ considered as a cocharacter of $H$.  Let $\{ \mu_i \}$ denote the conjugacy class of $\mu_i$.

The following theorem determines $R \Psi_{\mu}$ with the Galois action restricted to $\Gal(\overline{F}/\widetilde{F})$. 
\begin{thm} \label{convid} Let $\IC_{\mu_i}$ be the intersection cohomology sheaf on $S_H(\mu_i) \subset \Gr_{H_{\overline{k}}}$.  Then we have a natural isomorphism
$$
\widetilde{R \Psi}_{\mu} \cong \IC_{\mu_1} \star \IC_{\mu_2} \star \cdots \star \IC_{\mu_d}
$$
where $\star$ denotes the convolution product on $\Gr_{H_{\overline{k}}}$. 
\end{thm}

When $G = \GL_n$ or $\GSp_{2g}$ and $\mu$ is minuscule, Theorem \ref{convid} is proven in  \cite[Theorem 13.1]{PR2}.  Our proof follows a similar strategy.  We first introduce a version of the splitting models of Pappas and Rapoport \cite{PR2}. For each embedding $\psi_i:K \ra \overline{F}$, let $\varpi_i := \psi_i(\varpi)$ where $\varpi$ is the fixed uniformizer of $K$. Set $\widetilde{X} =  \bA^1_{\cO_{\widetilde{F}}}$, and let $D_i \subset \widetilde{X}$ be the principal divisor defined by $u - \varpi_i = 0$. 

\begin{defn} \label{splitmodel} Define the \emph{splitting Grassmanian} $\Split_H^{Q(u)}$ by the following functor on $\cO_{\widetilde{F}}$-algebras:
$$
\Split_H^{Q(u)}(R) := \{ \text{iso-classes of d-tuples } (\cE_i, \alpha_i) \}
$$
where $\cE_i$ is an $H$-bundle on $\widetilde{X}_R$ and $\alpha_i:(\cE_i)_{\widetilde{X}_R - D_i} \cong (\cE_{i + 1})_{\widetilde{X}_R  -  D_i}$ for $1 \leq i \leq d-1$ and $\alpha_d$ is a trivialization of $\cE_{d}$ along $\widetilde{X}_R - D_d$.  
\end{defn}

\begin{prop} \label{splitprop} The functor $\Split_H^{Q(u)}$ is represented by an ind-scheme over $\Spec \cO_{\widetilde{F}}$ which is ind-proper.
\end{prop}
\begin{proof} The proof is to the same as the proof of Proposition \ref{indproper3}. 
\end{proof}

There is a natural map 
$$
\widetilde{m}:\Split_H^{Q(u)} \ra (\Gr_{\cG}^{Q(u)})_{\cO_{\widetilde{F}}}  
$$
given by $\{ (\cE_i, \alpha_i) \} \mapsto (\cE_1, (\alpha_d \alpha_{d-1} \cdots \alpha_1)|_{X_R - D})$.

\begin{prop} \label{splitfibers} The morphism $\widetilde{m}$ induces an isomorphism on generic fibers over $\widetilde{F}$.  There is a natural isomorphism  
$$
\left(\Split_H^{Q(u)}\right)_{\widetilde{k}} \cong \Gr_{H_{\widetilde{k}}} \widetilde{\times} \Gr_{H_{\widetilde{k}}} \widetilde{\times} \cdots \widetilde{\times} \Gr_{H_{\widetilde{k}}},
$$
where the right side is the $d$-fold convolution product, such that $\widetilde{m}_k:\left(\Split_H^{Q(u)}\right)_{\widetilde{k}} \ra \Gr_H$ is the $d$-fold multiplication map.
\end{prop} 
\begin{proof} The argument as the same as in Proposition \ref{Convfibers}.
\end{proof}

The proof of Theorem \ref{convid} follows the same lines as the proof of Theorem 13.1 in \cite{PR2}. 
\begin{proof}[Proof of Theorem \ref{convid}]   We first observe that $\cG_{\widetilde{X}} \cong H_{\widetilde{X}}$ since $\fa$ is a hyperspecial vertex of $\cB(G, K)$ so we will work entirely with $H$. 

By Proposition \ref{splitfibers}, $\widetilde{m}_{\widetilde{k}}$ gives $d$-fold convolution.  Since $\widetilde{m}$ is ind-proper (or proper if we restrict to the preimage of $M_{\cG}(\mu)$), it suffices then to show that 
$$
R \Psi (\widetilde{m}_{\widetilde{F}}^*(\IC_{\mu}) ) \cong \IC_{\mu_1} \widetilde{\boxtimes} \IC_{\mu_2} \widetilde{\boxtimes} \cdots \widetilde{\boxtimes} \IC_{\mu_d}.
$$
Recall the following diagram:
\begin{equation} \label{qp1}
\xymatrix{
& \ar[dl]_{p} \ar[dr]^{q} L H^{d-1}_{\widetilde{k}} \times \Gr_{H_{\widetilde{k}}} & \\
\Gr_{H_{\widetilde{k}}} \times \cdots \times \Gr_{H_{\widetilde{k}}} & & \Gr_{H_{\widetilde{k}}} \widetilde{\times} \cdots \widetilde{\times} \Gr_{H_{\widetilde{k}}}
}
\end{equation}

The twisted product is defined as the unique (up to canonical isomorphism) sheaf  $\IC_{\mu_1} \widetilde{\boxtimes} \cdots \widetilde{\boxtimes} \IC_{\mu_d}$ such that 
$$
q^*(\IC_{\mu_1} \widetilde{\boxtimes} \cdots \widetilde{\boxtimes} \IC_{\mu_d}) = p^*(\IC_{\mu_1} \boxtimes \cdots \boxtimes \IC_{\mu_d}).
$$

One can deform the diagram (\ref{qp1}) to a diagram 
\begin{equation} \label{qp2}
\xymatrix{
& \ar[dl]_{\widetilde{p}} \ar[dr]^{\widetilde{q}}  \Split^{Q(u), \infty}_H& \\
\Gr_{H, \omega_1} \times \cdots \times \Gr_{H, \omega_d} & & \Split^{Q(u)}_H
}
\end{equation}
over $\Spec \cO_{\widetilde{F}}$ where $\Gr_{H, \omega_i}$ is the affine Grassmanian for $H$ centered at $u - \varpi_i$ and 
$\Split^{Q(u), \infty}_H$ is defined below.

Let $\widehat{D}_i$ denote the completion of $\widetilde{X}$ along $D_i$. Then $\Split^{Q(u), \infty}_H$ is the ind-scheme which represents the functor of isomorphism classes of $\{ (\cE_1, \alpha_1), (\cE_i, \alpha_i, \gamma_i)_{i \geq 2} \}$ with $\{ (\cE_i, \alpha_i) \} \in \Split^{Q(u)}_H$ and $\gamma_i$ a trivialization of $\cE_i$ along the completion of $\widehat{D}_{i-1}$.  The map $\widetilde{q}$ is the obvious one.  The map $\widetilde{p}$ is defined by 
$$
\{  (\cE_1, \alpha_1), (\cE_i, \alpha_i, \gamma_i)_{i \geq 2} \} \mapsto \{ (\cE_i, \gamma_{i+1}^* \alpha_i)_{i \leq d-1}, (\cE_d, \alpha_d) \}
$$
where $\gamma_{i+1}^*$ is the restriction of $\gamma_{i+1}$ to the punctured formal disc $\widehat{D}_i[1/(u- \varpi_i)]$. Compare this definition to that of $\widetilde{\Gr}_{\cG, X}$ on \cite[pg. 232]{PZ}.  

Since nearby cycles commute with smooth base change \footnote{The morphisms $\widetilde{q}$ and $\widetilde{p}$ are not smooth as written. However, for any particular $\mu$, all the relevant sheaves are supported on the finite type closed subschemes. Working over the support of $\IC_{\mu}$ we can replace $\Split^{Q(u), \infty}_H$ by a finite type smooth torsor where we take trivializations only over $n$th infinitessimal neighborhoods for some sufficiently large $n$.} and $R \Psi^{\Gr_{H, \varpi_i}} \IC_{\mu_i, \widetilde{F}} = \IC_{\mu_i, \widetilde{k}}$ with the trivial Galois action, by the same argument as in the proof of Theorem \ref{commconst} (or \cite[Proposition 10.7]{PZ}) we are reduced to showing that
\begin{equation} \label{qp5}
\widetilde{q}_{\widetilde{F}}^*(\widetilde{m}_{\widetilde{F}}^*(\IC_{\mu}) )) \cong \widetilde{p}_{\widetilde{F}}^* \left(\IC_{\mu_1, \widetilde{F}} \boxtimes \cdots \boxtimes \IC_{\mu_d, \widetilde{F}} \right).
\end{equation}
If we let $\theta:\left(\Gr^{Q(u)}_{\cG} \right)_{\widetilde{F}} \cong \prod_i \Gr_{H_{\widetilde{F}}, \omega_i}$ be the isomorphism from Proposition \ref{genfiber},  The composition $(\theta \circ \widetilde{m}_{\widetilde{K}})( \{ (\cE_i, \alpha_i) \}) = \{ (\cF_i, \beta_i) \}$ where $\cF_i$ is the completion of $\cE_i$ along $D_i$ and $\beta_i = \alpha_d \alpha_{d-1} \ldots \alpha_i$ is a  trivialization along $\widehat{D}_i[1/(u - \varpi_i)]$.  Note that each $\alpha_j$ is defined on $\widehat{D}_i[1/(u - \varpi)]$ and for $i \neq j$, $\alpha_j$ is defined on $\widehat{D}_i$ so $\alpha_d \alpha_{d-1} \ldots \alpha_{i+1}$ defines a trivialization of $\cE_{i+1}$ along $\widehat{D}_i$.   For $\{ (\cE_1, \alpha_1), (\cE_i, \alpha_i, \gamma_i)_{i \geq 2} \}$ in $\left( \Split^{Q(u), \infty}_H \right)_{\widetilde{F}}$, the composition $\alpha_d \alpha_{d-1} \ldots \alpha_{i+1}$ differs from $\gamma_{i+1}$ by an element of $L^+_{\omega_i} H_{\widetilde{F}}$.  Thus, the $\widetilde{p}_{\widetilde{F}}$ differs from the composition $\theta \circ \widetilde{m}_{\widetilde{F}} \circ \widetilde{q}_{\widetilde{F}}$ by multiplication by the product $\prod_i L^+_{\omega_i} H_{\widetilde{F}}$. Since each $\IC_{\mu_i}$ is $L^+_{\omega_i} H_{\widetilde{F}}$-equivariant the two pullbacks are isomorphic which implies (\ref{qp5}).  
\end{proof}
 
Finally, we would like to state a conjecture on the action of $I_E$ on $R \Psi_{\mu}$.   Geometric Satake for $H_{\overline{k}}$ gives an equivalence of tensor categories
$$
S_{\overline{k}}: \Perv_{L^+ H_{\overline{k}}} \Gr_{H_{\overline{k}}} \xrightarrow{\sim} \Rep_{\overline{\Q}_{\ell}} (H^{\vee})
$$
sending $\IC_{\la}$ to $V_{\la}$ the representation of $H^{\vee}$ with highest weight $\la$. Furthermore, $S_{\overline{k}}$ is realized by taking  (hyper)cohomology.   It suffices then to determine the action of $I_E$ on
\begin{equation} \label{h2}
H^*(\widetilde{R \Psi}_{\mu}) = V_{\mu_1} \otimes V_{\mu_2} \otimes \cdots \otimes V_{\mu_d}.
\end{equation}
As $G$ is unramified, the inertia subgroup $I_F$ of $\Gal(\widetilde{F}/F)$ acts through a subgroup of the permutation group on the embeddings $\{\psi_i\}$.  Furthermore, $I_E$ is the subgroup of $I_F$ of those $\sigma \in S_d$ such that $\mu_{\sigma(i)} = \mu_i$.  This group has a natural permutation action on $V_{\mu_1} \otimes V_{\mu_2} \otimes \cdots \otimes V_{\mu_d}$  which we call $\rho$.  

As in \cite[Remark 7.4]{PR1}, we have
\begin{equation} \label{s6}
V_{\mu_1} \otimes V_{\mu_2} \otimes \cdots \otimes V_{\mu_d} = \oplus_{\la \leq \mu_1 + \ldots + \mu_d} M_{\la} \otimes V_{\la}
\end{equation}
where $H^{\vee}$ acts trivially on $M_{\la}$.  Thus, $\rho$ decomposes as 
$$
\oplus_{\la \leq \mu_1 + \ldots + \mu_d} \rho_{\la} \otimes \mathrm{id}_{V_{\la}}  
$$
where $\rho_{\la}$ acts on $M_{\la}$.   

The following was conjectured by Pappas and Rapoport in \cite[Remark 7.4]{PR1} for $\GL_n$:
\begin{conj} \label{conject} There is a decomposition 
$$
R \Psi_{\mu} \cong \oplus_{\la \leq \mu_1 + \ldots \mu_d} M_{\la} \otimes \IC_{\la}
$$
where $M_{\la}$ the constant sheaf associated to the vector space $M_{\la}$ in Equation \ref{s6}. The action of $I_E$ is trivial on $\IC_{\la}$ and is isomorphic to $\rho_{\la}$ on $M_{\la}$.
\end{conj} 

\begin{rmk} When $K/F$ is tamely ramified, Conjecture \ref{conject} is a consequence of \cite[Theorem 10.23]{PZ}.  It may be possible to prove the conjecture by adapting the proof of Theorem 10.18 in \cite{PZ} to our setting, but we do not attempt to do so here. 
\end{rmk}

\end{document}